\documentclass[a4paper,reqno,11pt]{amsart}
\usepackage{t1enc}
\usepackage[margin=0.7in]{geometry}
\usepackage{ifthen}
\usepackage{graphicx}
\usepackage{amsmath,amstext,amssymb,amsopn,amsthm,color}
\usepackage{mathtools}

\newcommand {\lni}{\lambda_i^{(n)}}
\newcommand {\lnj}{\lambda_j^{(n)}}

\newcommand{\C}{\mathbf{C}}
\newcommand{\R}{\mathbf{R}}

\newcommand{\pr}{\mathbf{P}}
\newcommand{\ex}{\mathbf{E}}
\newcommand{\E}{\mathbf{E}}
\newcommand{\N}{\mathbf{N}}

\newcommand{\ind}{\mathbf{1}}

\theoremstyle{plain}
\newtheorem{theorem}{Theorem}
\newtheorem{lemma}{Lemma}

\newtheorem{proposition}{Proposition}

\theoremstyle{definition}

\newtheorem{remark}{Remark}

\theoremstyle{remark}

\DeclareMathOperator{\supp}{supp}

\newcommand{\formula}[2][nolabel]
{\ifthenelse{\equal{#1}{nolabel}}
 {\begin{align*} #2 \end{align*}}
 {\ifthenelse{\equal{#1}{}}
  {\begin{align} #2 \end{align}}
  {\begin{align} \label{#1} #2 \end{align}}
 }
}

\numberwithin{equation}{section}

\begin{document}

%
%

\title[Universality classes for general random matrix flows]{Universality classes for general random matrix flows}

\author{Jacek Ma{\l}ecki, Jos\'e Luis P\'erez}

\address{  Jacek Ma{\l}ecki,  \\  Faculty of Pure and Applied Mathematics\\ Wroc{\l}aw University of Science and Technology \\ ul. Wybrze{\.z}e Wyspia{\'n}\-skiego 27 \\ 50-370 Wroc{\l}aw, Poland}
\email{jacek.malecki@pwr.edu.pl }
\address{Jos\'e Luis P\'erez \\ Centro  de  Investigaci\'on  en  Matem\'aticas  A.C.,  Calle  Jalisco  s/n,  CP
36240, Guanajuato, Mexico}
\email{jluis.garmendia@cimat.mx}


\thanks{J. Ma{\l}ecki is supported by the Polish National Science Centre (NCN) grant no. 2018/29/B/ST1/02030.}

\begin{abstract} 
We consider matrix-valued processes described as solutions to stochastic differential equations of very general form. We study the family of the empirical measure-valued processes constructed from the corresponding eigenvalues. We show that the family indexed by the size of the matrix is tight under very mild assumptions on the coefficients of the initial SDE. We characterize the limiting distributions of its subsequences as solutions to an integral equation. We use this result to study some universality classes of random matrix flows. These generalize the classical results related to Dyson Brownian motion and squared Bessel particle systems. We study some new phenomenons as the existence of the generalized Marchenko-Pastur distributions supported on the real line. We also introduce universality classes related to generalized geometric matrix Brownian motions and Jacobi processes. Finally we study, under some conditions, the convergence of the empirical measure-valued process of eigenvalues associated to matrix flows to the law of a free diffusion.
\end{abstract}

\maketitle

\section{Introduction}

Since the pioneer work of Eugen Wigner \cite{bib:Wigner1955}, the study of the asymptotic behavior of empirical measures based on the eigenvalues of random matrices has given rise to a significant literature. Apart from random matrix theory, the problem has its importance in free probability, mathematical physics (chaotic quantum physics), principal component analysis in statistics, communication theory and even in number theory. 
In the present article we study this phenomenon from the stochastic point of view and prove the convergence of empirical measure-valued processes related to solutions of a very wide class of matrix SDEs, which, in particular, generalizes the known results for the Dyson Brownian motion and the Wishart processes. More precisely, we consider solutions $X=(X_t)$ to the general matrix valued stochastic differential equations on $\mathcal{H}_n$, the space of Hermitian $n\times n$ matrices, of the form
\begin{eqnarray}
 \label{eq:MSDE:hermitian}
 dX_t=g(X_t)dW_th(X_t)+h(X_t)dW_t^*g(X_t)+b(X_t)dt\/,\quad X_0\in \mathcal{H}_n\/,
\end{eqnarray}
where the continuous functions $g,h,b:\R\to\R$ act spectrally on $X_t$. Here $W=(W_t)$ stands for $n\times n$ complex-valued Brownian motion, i.e. the matrix valued process with entries being independent one-dimensional complex-valued Brownian motions. Continuity of the coefficients ensures the existence of solutions and the symmetric form of the martingale part makes them indeed elements of $\mathcal{H}_n$ (see \cite{bib:gm2013}).
The main goal of the paper is to study the asymptotic behavior of $X$ when $n$ goes to infinity, through the empirical measure-valued process based on the eigenvalues of $X$. As in the classical random matrix setting, it requires suitable normalization of the original process. Thus, we consider $X^{(n)}=(X^{(n)}_t)$, a solution of the scaled SDE of the form 
 \begin{eqnarray*}
 dX_t^{(n)}=g(X_t^{(n)})dW_t^{(n)} h(X_t^{(n)})+h(X_t^{(n)})d(W_t^{(n)})^*g(X_t^{(n)})+\frac{1}{n}b_n(X_t^{(n)})dt\/,\quad X_0\in \mathcal{H}_n\/,
\end{eqnarray*}
where $W^{(n)}_t=\frac{1}{n^{1/2}}W_{t}$ and the drift term $b_n$ depends on $n$. Note that this approach is equivalent to considering the time-scaled process $(X_{t/n})$, where $X$ is a solution to the original \eqref{eq:MSDE:hermitian} with $b$ replaced by $b_n$. The appearance of the drift $n^{-1}b_n(X_t^{(n)})$ is natural in the view of the classical result for the Wishart processes (see \cite{bib:CabanalGuionnet:2001}). However, we can and we do go beyond this and consider
 \begin{eqnarray}
 \label{eq:MSDE:hermitian:scaled}
 dX_t^{(n)}=g_n(X_t^{(n)})dW_t^{(n)} h_n(X_t^{(n)})+h_n(X_t^{(n)})d(W_t^{(n)})^*g_n(X_t^{(n)})+\frac{1}{n}b_n(X_t^{(n)})dt\/,\quad X_0\in \mathcal{H}_n\/,
\end{eqnarray}
where the martingale coefficients also vary with the growth of the dimension $n$. Let us now denote by $\lambda_1^{(n)}\leq \lambda_2^{(n)}\leq \ldots\leq\lambda_n^{(n)}$ the ordered eigenvalues of $X^{(n)}$ given by \eqref{eq:MSDE:hermitian:scaled} and define the corresponding empirical measure-valued process
\begin{eqnarray}
   \label{eq:EM:defn}
   \mu_t^{(n)}(dx) = \dfrac{1}{n}\sum_{i=1}^n \delta_{\lambda_i^{(n)}(t)}(dx)\/,
\end{eqnarray}
where $\delta_a$ is the unit mass at $a\in\R$. We study the convergence of $(\mu_t^{(n)})$, as $n$ goes to infinity, in probability in $\mathcal{C}(\R_{+},\mathrm{Pr}(\R))$, i.e. the space of continuous functions equipped with the topology of uniform convergence on closed intervals, which take values in the space of probability measures $\mathrm{Pr}(\R)$ with the topology of weak convergence. 

Our results cover also the real-valued case, i.e. starting with the following equation
\begin{eqnarray}
\label{eq:MSDE:real}
 dX_t=g(X_t)dB_th(X_t)+h(X_t)dB_t^Tg(X_t)+b(X_t)dt\/,\quad X_0\in \mathcal{S}_n\/,
\end{eqnarray}
where $B$ is real Brownian motion matrix and $B^T$ stands for its transpose, we obtain a solution $X$ in the space of symmetric $n\times n$ matrices $\mathcal{S}_n$. We define the empirical measure process as in \eqref{eq:EM:defn}, where $\lambda_i^{(n)}$ are now eigenvalues of a solution to the scaled SDE of the form
 \begin{eqnarray}
 \label{eq:MSDE:real:scaled}
 dX_t^{(n)}=g_n(X_t^{(n)})dB_t^{(n)} h_n(X_t^{(n)})+h_n(X_t^{(n)})d(B_t^{(n)})^Tg_n(X_t^{(n)})+\frac{1}{n}b_n(X_t^{(n)})dt\/,\quad X_0\in \mathcal{S}_n\/,
\end{eqnarray}
where $B^{(n)}_t=\frac{1}{n^{1/2}}B_{t}$. We will deal with both cases in a unified way by introducing the parameter $\beta$, with the convention that $\beta=1$ in the real-valued case and $\beta=2$ in the complex-valued case.

The main result of the article is the following
\begin{theorem}
	\label{thm:main} 
	Assume that for continuous functions $g_n$, $h_n$ and $b_n$ there exists constant $K>0$ such that
	\begin{eqnarray}
	  \label{eq:growth}
	  g^2_n(x)+h^2_n(x)\leq K(1+|x|)\/,\quad \frac{|b_n(x)|}{n}\leq K(1+|x|)
	\end{eqnarray}
	for every $x\in \R$ and $n\in \N$. If additionally 
	\begin{equation}
	   \label{eq:8momentbound}
	   \sup_{n}\int_{\R}x^8\mu_0^{(n)}(dx)<\infty\/,
	\end{equation}
 then the family of the measure-valued processes $\{(\mu_{t}^{(n)})_{t\geq0}:n\geq1\}$ related to a solution of \eqref{eq:MSDE:hermitian:scaled} or \eqref{eq:MSDE:real:scaled} is tight. If $(\mu_{t})_{t\geq0}$ is the limit of its weakly convergent subsequence in $\mathcal{C}(\R_{+},\mathrm{Pr}(\R))$ and $g_n^2(x)\to g^2(x)$, $h_n^2(x)\to h^2(x)$, $b_n(x)/n\to b(x)$ (locally uniformly on $\R$), then $(\mu_{t})_{t\geq0}$ is the continuous probability-measure valued function satisfying
	\begin{equation} 
	\langle\mu_{t},f\rangle=\langle \mu_0,f\rangle+\int_{0}^{t}ds\int_{\R}b(x)f'(x)\mu_s(dx)+\frac{\beta}{2} \int_{0}^{t}ds\int_{\R^{2}}
	\frac{f^{\prime}(x)-f^{\prime}(y)}{x-y}G(x,y)\mu_{s}(dx)\mu_{s}(dy),
	\label{limit}
	\end{equation}
	for each $t\geq0$, $f\in \mathcal{C}_{b}^{2}(\R)$, where $G(x,y)=g^2(x)h^2(y)+g^2(y)h^2(x)$, $\beta=2$ in the complex-valued case and $\beta=1$ in the real-valued case. 
	
	If \eqref{limit} has the unique solution, then  $\{(\mu_{t}^{(n)})_{t\geq0}:n\geq1\}$ converges to $(\mu_t)_{t\geq 0}$ in $\mathcal{C}(\R_{+},\mathrm{Pr}(\R))$.
\end{theorem}

\medskip

Here, for any $\mu\in\mathrm{Pr}(\R)$, $\langle\mu,f\rangle = \int fd\mu $ and we understand $(f'(x)-f'(y))/(x-y)$ to be a continuous function on $\R^2$, i.e. it is equal to $f''(x)$ for $x=y$. 

\begin{remark}
   \label{rem:classf}
  Under given regularity of the coefficients, equation \eqref{limit} holds in fact for every $f\in\mathcal{C}^2(\R)$ such that $f$, $f'$ and $f''$ have sub-polynomial growth as it shown in Section \ref{sec:moments}.
\end{remark}

\begin{remark}
   As it is studied in detail in Section \ref{sec:Wishart}, the uniqueness of the solution to \eqref{limit} does not hold in general. Even the Lipschitz continuity of $g$, $h$ and $b$ do not ensure uniqueness. 
\end{remark}

\begin{remark}
Note also that we do not assume any conditions ensuring uniqueness of a solution of \eqref{eq:MSDE:hermitian:scaled}. Since continuity of the coefficients implies existence, we show that any sequence of empirical measure processes is relatively compact. Moreover if the equation \eqref{limit} has a unique solution then the convergence holds for any sequence.
 The sub-linear growth conditions on $g^2_n(x)$, $h^2_n(x)$ and $b_n(x)/n$ ensure that the possible explosion time is infinite a.s. (see Lemma \ref{lem:pk:bound} below). The assumed continuity together with \eqref{eq:growth} seem to be very weak assumptions and thus Theorem \ref{thm:main} covers a very wide range of general SDEs of the form given in \eqref{eq:MSDE:hermitian:scaled}.
\end{remark}

\begin{remark}
   We emphasize that we do not assume that the eigenvalues at the initial point are distinct and, what is more important, we do not impose any conditions on the coefficient to prevent the eigenvalues from colliding after the start. The natural way to study $(\mu_t^{(n)})$ is to write the SDEs for $\lambda_i^{(n)}$ (see \eqref{eq:eigen:SDE} below), which becomes troublesome at the first collision time of the eigenvalues because of the expressions $(\lambda_i^{(n)}-\lambda_j^{(n)})^{-1}$ appearing in the drift terms. Thus non-collision results related to particular models were very often the crucial parts of the proofs and caused some artificial restrictions and assumptions. Note that the eigenvalue process and consequently the empirical measure-valued process are always well-defined whenever $X$ exists. The question about its asymptotic behavior can be asked independently from existence or nonexistence of collisions between eigenvalues. 
Thus the restriction to non-colliding systems has often been the main technical issue and this obstacle is completely removed in our approach.
\end{remark}

We also study the convergence of moments of $(\lambda_1^{(n)},\ldots,\lambda^{(n)}_n)$, i.e. the behavior of
\begin{equation*}
    \frac{1}{n}\ex\left((\lambda_1^{(n)}(t))^m+\ldots+(\lambda^{(n)}_n(t))^m\right) = \int_{\R}x^m\mu_t^{(n)}(dx) \/,
\end{equation*}
as $n$ goes to infinity. We show that some properties of the moments of $\mu_t$ are inherited from the initial distributions $\mu_0^{(n)}$ and $\mu_0$. These results are given in the following theorem and its proof is postponed until Section \ref{sec:moments}. The properties of the moments will be used to show uniqueness of solutions to \eqref{limit} in the study of certain universality classes, i.e. for specific choice of $g$, $h$ and $b$.

\begin{theorem}
  \label{thm:moments}
	Let $(\mu_t^{(n)})_{t\geq 0}$ be a family the empirical measure-valued processes defined for \eqref{eq:MSDE:hermitian:scaled} or \eqref{eq:MSDE:real:scaled} with the continuous coefficients $g_n$, $h_n$ and $b_n/n$ fulfilling \eqref{eq:growth} and such that $g_n^2$, $h_n^2$ and $b_n/n$ are locally uniformly convergent. Let $(\mu_t)_{t\geq 0}$ be a weak limit in $\mathcal{C}(\R_+,\textrm{Pr}(\R))$ of the subsequence $(\mu_t^{(n_i)})_{t\geq 0}$. If 
	\begin{equation}
	   \label{eq:momentbounds:0}
		  \sup_{i} \int_{\R}x^{2k}\mu_0^{(n_i)}(dx)<\infty\/,
	\end{equation}
	for $k\in\N$, then for every $T>0$ we have
	\begin{equation}
	    \label{eq:momentbounds:t}
        \sup_{t\leq T}\int_{\R}x^{2k}\mu_t(dx)<\infty\/,
	\end{equation}
	and  
	\begin{equation*}
	   \int_{\R} x^m \mu^{(n_i)}_t(dx) \to \int_{\R}x^m\mu_t(dx)\/,\quad i\to\infty\/,
	\end{equation*}
	for every $m=0,1,\ldots,2k-1$. 
	
	If $(\mu_t)_{t\geq 0}$ is a solution to \eqref{limit} with $g^2(x)+h^2(x)\leq K(1+|x|)$ and $b$ is bounded on $[0,\infty)$, non-negative on $(-\infty,0)$ and $\mu_0$ has a characteristic function, which is analytic on a neighborhood of the origin, then, for every $t>0$, the characteristic function of $\mu_t$ is analytic on a neighborhood of the origin and, in particular, it is uniquely determined by its moments.
\end{theorem}

\medskip

We use these results to study some universality classes for random matrix flows. We consider four main examples, in which Theorem \ref{thm:main} can be applied. The first is introduced in Section \ref{sec:Wigner} and it relates to the Wigner ensemble, where we consider a very wide class of matrix flows, which leads to the family of Wigner's semi-circle laws $(\rho_t^{sc})$ as the limit of the corresponding empirical measure-valued processes. The result is given in Theorems \ref{thm:Wigner1} and \ref{thm:Wigner2}, which are generalizations of Rogers and Shi's result from \cite{bib:RogersShi:1993}.  Then, in Section \ref{sec:Wishart}, we consider the generalized Laguerre/Wishart processes and the corresponding integral equation \eqref{limit}. We show that the uniqueness of solutions does not hold in this case by introducing new families of measures, which can be considered as the generalized Marchenko-Pastur distributions. Although such distributions can be obtained as limits of the empirical measures associated to certain matrix flows, we provide in Theorem \ref{thm:Wishart:class} additional conditions on the coefficients of \eqref{eq:MSDE:hermitian} to ensure convergence of the empirical measure-valued processes to the family of the Marchenko-Pastur distributions $(\mu_t^{MP})$. This significantly generalizes the result of Cabanal Duvillard and Guionnet from \cite{bib:CabanalGuionnet:2001}. In Section \ref{sec:GMBM} we consider the matrix-analogue of the geometric Brownian motion. We show the convergence and characterize the moments of the family of limiting distributions $(\mu_t^{geo})$ in Theorem \ref{thm:GMBM:1}. Then we introduce the general result in Theorem \ref{thm:GMBM:class}, where we provide a wide class of solutions to SDE's of the form \eqref{eq:MSDE:hermitian:scaled}, which leads to  $(\mu_t^{geo})$. Finally, in Section \ref{sec:Jacobi}, we consider the convergence of the empirical measure-valued processes for Jacobi processes and introduce the universality class associated to the family of limiting distributions $(\mu_t^{Jac})$. Although all the results are stated for measure-valued processes, by fixing the time variable $t$, one can obtain the weak convergence of the empirical measures for random matrices.  

\medskip

The relation between our results and free probability is studied in Section \ref{sec:Free}. We start by introducing the notion of a free diffusion, which is the solution to a free stochastic differential equation of the form:
\begin{align*}
dX_t=g(x_t)dZ_th(X_t)+h(X_t)dZ^*_tg(X_t)+b(X_t)dt,\qquad\text{$X_0\in\mathcal{A}$,}
\end{align*}
where $Z$ is a complex free Brownian motion, $g,h,b:\R\to\R$ act spectrally on $X$ and are assumed to be locally operator Lipschtiz continuous. We then study the convergence of the empirical measured-valued process of eigenvalues associated to matrix flows to the law of the free diffusion $(X_t)$. This is done by showing that the Cauchy transforms of the law of the free diffusion $(X_t)$ and the limit of the empirical measured-valued process of the associated matrix flows, given in Theorem \ref{thm:main}, satisfy the same differential equation. Hence, if there is uniqueness to this differential equation, then the limiting law of the sequence of empirical measure processes associated to random matrix flows and the law of the free diffusion must be the same. We finally illustrate this results by providing some examples in which we can show this convergence: the free linear Brownian motion, and the free Ornstein-Uhlenbeck process.
\section{Symmetric polynomials}
\label{sec:sympol}
\subsection{Preliminaries}
Recall that we denote by $\mathcal{S}_n$ the space of real symmetric $n\times n$ matrices and $\mathcal{H}_n$ stands for Hermitian $n\times n$ matrices. In both cases, for a given function $g:\R\to\R$ and a matrix $X$, we write $g(X)$ for the spectral action of $g$ on $X$, i.e.  
\begin{eqnarray*}
g(X)=H\text{diag}[g(\lambda_1),\dots,g(\lambda_n)]H^T,
\end{eqnarray*}
where $X=H\Lambda H^T$ is a diagonalization of $X$ with an orthogonal matrix $H$ and an eigenvalue matrix $\Lambda=\text{diag}[\lambda_1,\dots,\lambda_n]$. For a given symmetric or Hermitian matrix $X$ we define the related elementary symmetric polynomials
\begin{eqnarray}
\label{eq:poly:defn}
e_k = e_k(X) := \sum_{i_1<\ldots<i_k}\lambda_{i_1}\cdot\ldots\cdot \lambda_{i_k}\/,\quad i=1,\ldots,n\/,
\end{eqnarray}
in the eigenvalues $\lambda_{1},\ldots,\lambda_{n}$ of the matrix $X$. We also use the convention that $e_0\equiv 1$. The elementary symmetric polynomials are, up to the sign change, the coefficients of the characteristic polynomial of $X$
\begin{eqnarray*}
  \det (X-u I) = \sum_{k=0}^n (-1)^{n-k}u^{n-k}e_k(X)\/.
\end{eqnarray*}
In particular, $e_n(X)$ are polynomial functions of the entries of $X$. By the fundamental theorem of symmetric polynomials we get that every symmetric polynomial in $\lambda_1,\ldots,\lambda_n$ has a unique representation as a polynomial of $e_1,\ldots,e_n$.

We will also use the notation $e_n^{\overline{\lambda_i}}$ for the incomplete polynomial of order $n$, not containing the variable $\lambda_i$. The notation $e_n^{\overline{\lambda_i}, \overline{\lambda_j}}$ is analogous, i.e. it stands for the polynomial of degree $n$ which does not contain $\lambda_i$ and $\lambda_j$. Finally we set $e_0^{\overline{\lambda_i}}\equiv 1$ and $e_{-1}^{\overline{\lambda_i}, \overline{\lambda_j}}=0$. For every $n\in\N$ and $k=1,2,\ldots$ we write
\begin{eqnarray*}
  p_k = \sum_{i=1}^n\lambda_i^k
\end{eqnarray*}
for the power sum symmetric polynomials in the eigenvalues. Note that $p_k$ are related to $e_k$ by the following recurrence relation
\begin{eqnarray}
  \label{eq:pk:en}
  p_k=\sum_{i=1}^{k-1} (-1)^{i-1}e_ip_{k-i}+(-1)^{k-1}ke_k\/.
\end{eqnarray}

\subsection{Symmetric polynomials as continuous semimartingales}
In this section we assume that $X$ is a solution to \eqref{eq:MSDE:hermitian} or \eqref{eq:MSDE:real}. The basic symmetric polynomials related to $X$ are smooth functions of the coefficients and, by It\^o's formula, they are continuous semimartingales. Since $p_k$ are polynomial functions of $e_1,\ldots,e_n$, the same statement is true for $p_k$ and in fact for every other symmetric polynomial. In the next two lemmas we provide the semi-martingale description of $e_i$ and $p_k$. Note once again that $\lambda_1,\ldots,\lambda_n$ are continuous functions of $e_1,\ldots,e_n$ and also the entries of $X$ and we do not claim that they are semimartingales.

\begin{lemma}
\label{lem:sympol:SDE}
  Let $X$ be a solution to \eqref{eq:MSDE:hermitian} or \eqref{eq:MSDE:real}. Then $e_1,\ldots,e_n$ are semimartingales given by 
	\begin{eqnarray}
	  \label{eq:sympol:SDE}
	  de_k = \sum_{i=1}^n 2g(\lambda_i)h(\lambda_i)e_{k-1}^{\overline{\lambda_i}}d\nu_i + \left(\sum_{i=1}^n b(\lambda_i)e_{k-1}^{\overline{\lambda_i}}-\beta\sum_{i<j}G(\lambda_i,\lambda_j)e_{k-2}^{\overline{\lambda_i},\overline{\lambda_j}}\right)dt\/,
	\end{eqnarray}
	where $\nu_1,\ldots,\nu_n$ are independent Brownian motions and $k=1,\ldots,n$.
\end{lemma}
\begin{proof}
  The proof is basically the same as in Proposition 2.1 in \cite{bib:gmm2018}, where the special case of Wishart processes was considered. For the convenience of the reader we repeat this argument here. As we have mentioned, $e_1,\ldots,e_n$ are semimartingales, since they are given by an analytic mapping from $\mathcal{S}_n$ (or $\mathcal{H}_n$) to $\R^n$
	\begin{eqnarray*}
	  X\longrightarrow (e_1(X),\ldots,e_n(X))
	\end{eqnarray*}
	and let us denote this function as $F$. Whenever $\lambda_1(0)<\ldots<\lambda_n(0)$, i.e. there are no collisions at the starting point, and it was shown in Theorem 3 and Theorem 4 in \cite{bib:gm2013} that
	\begin{eqnarray}
	   \label{eq:eigen:SDE}
		  d\lambda_i = 2g(\lambda_i)h(\lambda_i)d\nu_i + b(\lambda_i)dt+ \beta\sum_{i\neq j} \dfrac{G(\lambda_i,\lambda_j)}{\lambda_i-\lambda_j}dt
	\end{eqnarray}
	up to the first collision time $T_c$. As it was calculated in the proof of Proposition 3.1 in \cite{bib:gm2014}, one can easily obtain 
	\eqref{eq:sympol:SDE} from \eqref{eq:eigen:SDE} in this case. To finish the proof and conclude that \eqref{eq:sympol:SDE} holds without any assumption on collisions between eigenvalues note that It\^o theorem states that the semimartingale representation of $e_1,\ldots,e_n$ is given in terms of the derivatives of the smooth function $F$. Moreover, we have just find these derivatives on the open set $\{X: \lambda_i(X)\neq \lambda_j(X), \textrm{for all} \ i\neq j, i,j=1,\ldots,n\}$. Since the exploding terms $(\lambda_i-\lambda_j)^{-1}$ are no longer present in \eqref{eq:sympol:SDE}, we can end the proof by a continuity argument. 
	\end{proof}
	
\begin{lemma}
\label{lem:pk:SDE}
  Let $X$ be a solution to \eqref{eq:MSDE:hermitian} or \eqref{eq:MSDE:real}. Then $p_k$ is a semimartingale, for every $k=1,2,\ldots$, described by
	\begin{eqnarray}
	  \label{eq:pk:SDE}
	dp_k &=& 2k\sum_{i=1}^n \lambda_i^{k-1}g(\lambda_i)h(\lambda_i)d\nu_i+2k(k-1)\sum_{i=1}^n \lambda_i^{k-2}g^2(\lambda_i)h^2(\lambda_i)dt\\
	&&+k\sum_{i=1}^n\lambda_i^{k-1}b(\lambda_i)dt+ \beta\,k\sum_{i<j}\left(\sum_{l=0}^{k-2}\lambda_i^{l}\lambda_j^{k-2-l}\right) G(\lambda_i,\lambda_j)dt\/,
	\end{eqnarray}
	where $\nu_i$, $i=1,\ldots,n$ is a collection of independent one-dimensional Brownian motions.
\end{lemma}
\begin{proof}
   This result can  be obtained directly from \eqref{eq:sympol:SDE} by using \eqref{eq:pk:en}, which requires more involved calculations. The other way is to find \eqref{eq:pk:SDE} using \eqref{eq:eigen:SDE} and assuming that there are no collision at the starting point and $t<T_c$. Note that the symmetry of $p_k$ and the symmetry of $G(x,y)$ will make the problematic term $(\lambda_i-\lambda_j)^{-1}$ disappear in the following way
	\begin{eqnarray*}
	   \sum_{i=1}^n \lambda_i^{k-1}\sum_{j\neq i} \frac{G(\lambda_i,\lambda_j)}{\lambda_i-\lambda_j} = \sum_{i<j}\frac{\lambda_i^{k-1}-\lambda_j^{k-1}}{\lambda_i-\lambda_j}\,G(\lambda_i,\lambda_j) = \sum_{i<j}\left(\sum_{l=0}^{k-2}\lambda_i^{l}\lambda_j^{k-2-l}\right) G(\lambda_i,\lambda_j)\/.
	\end{eqnarray*}
	Then the same continuity argument as in the proof of Lemma \ref{lem:sympol:SDE} ends the proof.
\end{proof}

\begin{remark}
\label{rem:pkn:SDE}
The results presented in Lemmas \ref{lem:sympol:SDE} and \ref{lem:pk:SDE} can be easily translated for the scaled processes and related symmetric polynomials and power sums by simple scaling of the coefficients. More precisely, we should replace $g$, $h$ and $b$ in the given semimartingale representations by $g_n/n^{1/4}$, $h_n/n^{1/4}$ and $b_n/n$ respectively. In particular, the power sums polynomials $p_k^{(n)}$ related to the solutions of scaled SDEs \eqref{eq:MSDE:hermitian:scaled} or \eqref{eq:MSDE:real:scaled} are described as
	\begin{eqnarray*}
	   dp_k^{(n)} &=& \frac{2k}{n^{1/2}}\sum_{i=1}^n (\lni)^{k-1}g_n(\lni)h_n(\lni)d\nu_i+\frac{2k(k-1)}{n}\sum_{i=1}^n (\lni)^{k-2}g^2_n(\lni)h^2_n(\lni)dt\\
	&&+\frac{k}{n}\sum_{i=1}^n(\lni)^{k-1}b_n(\lni)dt+ \beta\,\frac{k}{n}\sum_{i<j}\left(\sum_{l=0}^{k-2}(\lni)^{l}(\lnj)^{k-2-l}\right) G_n(\lni,\lnj)dt\/,
	\end{eqnarray*}
for given $k\in\N$, where $G_n(x,y)=g_n^2(x)h_n^2(y)+g_n^2(y)h_n^2(x)$. Note that the Brownian motions $\nu_i$ also depend on $n$, but since this is of no importance for the further computations (see Proposition \ref{prop:conv_martin}), we will not indicate it to make the notation simpler.
\end{remark}
In the next lemma we introduce very useful bounds for the second and forth moments of $\lambda_i^{(n)}$. In particular, the result implies that there are no explosion of solutions whenever \eqref{eq:growth} holds.
\begin{lemma}
\label{lem:pk:bound}
Let $p_k^{(n)}$ be the power sum polynomials related to  \eqref{eq:MSDE:hermitian:scaled} or \eqref{eq:MSDE:real:scaled} and assume that \eqref{eq:growth} holds. Then for given $T>0$ and $k=2,4,6,\ldots$ there exists a constant $C_1=C_1(T,K,k)>0$ such that 
\begin{eqnarray}
\label{eq:pk:bound}
   \E p_k^{(n)}(t)\leq C_1\left(n+p_k^{(n)}(0)\right)\/,
\end{eqnarray} 
for every $t<T$ and $n\in\N$.
\end{lemma}

\begin{proof} 
Let $T_m\nearrow \infty$ be a sequence of stopping times such that $M_k^{(n)}(t\wedge T_m)$ is a martingale, where
$$
  M_k^{(n)} = \frac{2k}{n^{1/2}}\sum_{i=1}^n (\lni)^{k-1}g_n(\lni)h_n(\lni)d\nu_i\/,
$$
and define $\tau_m= T_m\wedge \inf\{t\geq 0: p_k^{(n)}(t)\leq m\}$. Consequently, by Remark \ref{rem:pkn:SDE}, we obtain
\begin{eqnarray*}
  \left|\E p_k^{(n)}(t\wedge \tau_m)\right| &\leq& |p_k^{(n)}(0)|+\frac{k}{n}\E\sum_{i=1}^n \int_0^{t\wedge\tau_m} |\lni|^{k-1}|b_n(\lni)|ds\\
	&&+\frac{k\beta}{n}\E\int_0^{t\wedge\tau_m}\sum_{i\leq j}\left( \sum_{l=0}^{k-2}|\lni|^{l}|\lnj|^{k-2-l}\right)G_n(\lni,\lnj)ds\/.
\end{eqnarray*}
By \eqref{eq:growth}, we have
\begin{eqnarray*}
	\sum_{i=1}^n (\lni)^{k-1}\frac{|b_n(\lni)|}{n} &\leq& K\sum_{i=1}^n (|\lni|^{k-1}+|\lni|^k)\leq 2K(n+\sum_{i=1}^n|\lni|^k)
	\end{eqnarray*}
and	since for every $x,y\geq 0$ we have 
$$
  \sum_{l=0}^{k-2}x^{l}y^{k-2-l}\leq (k-1)(x^{k-2}+y^{k-2})\/,\quad (x^{k-2}+y^{k-2})(1+x)(1+y)\leq 8(1+x^k+y^k)
$$
we can estimate
$\frac{1}{n}\sum_{i\leq j}\left( \sum_{l=0}^{k-2}|\lni|^{l}|\lnj|^{k-2-l}\right)G_n(\lni,\lnj)$ from above by
	\begin{eqnarray*}
	\frac{16(k-1)K^2}{n}\sum_{i\leq j}((\lni)^{k-2}+(\lnj)^{k-2})(1+|\lni|)(1+|\lnj|)
	&\leq& \frac{c_1}{n}\sum_{i\leq j}(1+(\lni)^{k}+(\lnj)^k)\\
	&\leq& 2c_1 (n+\sum_{i=1}^n (\lni)^k)\/,
\end{eqnarray*}
where $c_1=c_1(K,k)$. Collecting all together we arrive at
\begin{eqnarray*}
   \E p_k^{(n)}(t\wedge \tau_m)\leq p_k^{(n)}(0)+ c_2T\,n +c_2\int_0^t \E p_k^{(n)}(s\wedge \tau_m)ds\/,
\end{eqnarray*}
for some $c_2=c_2(k,K)>0$ and every $t$ smaller than the fixed $T>0$.
The function $t\to \E p_k^{(n)}(t\wedge \tau_m)$ is continuous by the Dominate Convergence Theorem and the definition of $\tau_m$. Consequently, by the Gronwall's lemma, we get
\begin{eqnarray*}
   \E p_k^{(n)}(t\wedge \tau_m)\leq (p_k^{(n)}(0)+ c_2T\,n)e^{c_2T}\leq C_1(n+p_k^{(n)}(0))
\end{eqnarray*}
for $C_1=C_1(T,K,k)>0$. Finally, letting $m\to\infty$ and using the Fatou's lemma we obtain \eqref{eq:pk:bound}.
\end{proof}


\section{Proof of Theorem \ref{thm:main}}
The proof of Theorem \ref{thm:main} is divided into several parts given in a series of propositions. We begin by showing the semimartingale representation of the empirical measure process acting on smooth functions. We apply the results from Section \ref{sec:sympol} to show Proposition \ref{prop:EM:SDE} without additional assumptions about collisions of eigenvalues. Next we study the martingale part of \eqref{eq:EM:SDE} together with the extra drift term appearing only in the real-valued case and  we show that both of them vanish when $n$ goes to infinity (Propositions \ref{prop:conv_martin} and \ref{prop:conv-sec-der}). In Proposition \ref{prop:tightness} we show that the family of measures $\{(\mu_t^{(n)})_{t\geq 0}, n\in\N\}$ is tight, which is the last step to show weak convergence along subsequences. We conclude this section with a characterization of the weak limits.

Recall that in this section we assume that $\mu^{(n)}_t$ is defined for a solution $X^{(n)}$ of \eqref{eq:MSDE:hermitian:scaled} or \eqref{eq:MSDE:real:scaled}, where \eqref{eq:growth} holds. Moreover, we assume that \eqref{eq:8momentbound} holds.

\subsection{Semimartingale representation of the empirical measure}
\begin{proposition}
\label{prop:EM:SDE}
   For every $f\in\mathcal{C}_b^2(\R)$ we have
	\begin{align}
		  \langle\lefteqn{\mu^{(n)}_t,f\rangle =\langle\mu^{(n)}_0,f\rangle+\frac{2}{n^{3/2}}\sum_{i=1}^n\int_0^tf'(\lambda^{(n)}_i)g_n(\lambda_i^{(n)})h_n(\lambda_i^{(n)})d\nu_i+\frac{1}{n}\int_0^t\int_{\R}f'(x)b_n(x)\mu_s^{(n)}(dx)ds}\nonumber\\	 \label{eq:EM:SDE}  
			&+\frac{2-\beta}{2n}\int_0^t\int_{\R} f''(x)G_n(x,x)\mu_s^{(n)}(dx)ds +\frac{\beta}{2}\int_0^t\int_{\R^2}\frac{f'(x)-f'(y)}{x-y}\,G_n(x,y)\mu_s^{(n)}(dx)\mu_s^{(n)}(dy)ds\/,
	\end{align}
	where $\nu_i$ are independent Brownian motions.
\end{proposition}
\begin{proof}
  We begin by considering $f(x)=x^k$. Although this function does not belong to $\mathcal{C}_b^2(\R)$ we simply have that $\langle \mu^{(n)}_t,f\rangle = \frac{p_k^{(n)}(t)}{n}$. Moreover, 
	$$
	  f'(x)=k x^{k-1}\quad\quad \textrm{and}\quad\quad\dfrac{f'(x)-f'(y)}{x-y} = k \sum_{l=0}^{k-2}x^ly^{k-2-l}\/,
	$$
	which give
	\begin{eqnarray*}
	   \frac{\beta k}{n^2}\sum_{i<j}\lefteqn{\left(\sum_{l=0}^{k-2}x^ly^{k-2-l}\right) G_n(\lni,\lnj) = \frac{\beta}{n^2}\sum_{i<j}\frac{f'(\lni)-f'(\lnj)}{\lni-\lnj} G_n(\lni,\lnj)}\\
		&=&\frac{\beta}{2n^2}\sum_{i,j=1}^n \frac{f'(\lni)-f'(\lnj)}{\lni-\lnj} \,G_n(\lni,\lnj)- \frac{\beta}{2n^2}\sum_{i=1}^n f''(\lni)G_n(\lni,\lni)\\
		&=& \frac{\beta}{2}\int_{\R^2}\frac{f'(x)-f'(y)}{x-y} \,G_n(x,y)\mu_t^{(n)}(dx)\mu_t^{(n)}(dy)-\frac{\beta}{2n}\int_{\R}f''(x)G_n(x,x)\mu_t^{(n)}(dx)\/,
	\end{eqnarray*}
	where we have used a convention that $(f'(x)-f'(y))/(x-y)$ is equal to $f''(x)$ for $x=y$. Thus, by Remark \ref{rem:pkn:SDE}, we claim that \eqref{eq:EM:SDE} holds for $f(x)=x^k$ and consequently for every polynomial.
	
	Now fix a function $f\in \mathcal{C}_b^2(\R)$ and a constant $M=M(n)>0$ such that 
	$$
	-M<\lambda_1^{(n)}(0)\leq \ldots\leq \lambda_n^{(n)}(0)<M\/.
	$$ 
	There exists a sequence of polynomials $(f_k)_{k\in\N}$ such that $f_k \to f$, $f_k' \to f_k'$ and $f_k''\to f''$ uniformly on $[-M,M]$. In particular they are uniformly bounded on the interval. Let $\tau_M = \inf\{t: \lambda_1(t)=-M \textrm{ or } \lambda_n(t)=M\}$. Since \eqref{eq:EM:SDE} holds for every $f_k$ and $t$ changed into $t\wedge \tau_M$, we can apply the Dominated Convergence Theorem to show that  \eqref{eq:EM:SDE} is true for $f$ and every $t<\tau_M$. We finish the proof taking $M\to \infty$. Note that then $\tau_M\to \infty$ a.s. since there is no explosion of $\lni$ in finite time. 
\end{proof}

\begin{remark}
  Note that \eqref{eq:EM:SDE} holds without the growth condition \eqref{eq:growth}, but then we have to consider it up to the explosion time, which can be finite with positive probability.
\end{remark}

\subsection{Convergence of the martingale and the second derivative parts}

We begin with showing that the martingale part in \eqref{eq:EM:SDE} vanishes when the size of the matrix grows to infinity.
\begin{proposition} 
  \label{prop:conv_martin}
For every $T>0$ and $f\in\mathcal{C}_b^2(\R)$ we have that
	\begin{eqnarray*}
	\lim_{n\to\infty}\frac{2}{n^{3/2}}\sup_{t\in[0,T]}\sum_{i=1}^n\int_0^tf'(\lambda^{(n)}_i)g_n(\lambda_i^{(n)})h_n(\lambda_i^{(n)})d\nu_i &=& 0\qquad\text{a.s.}
	\end{eqnarray*}
\end{proposition}
\begin{proof}
By Doob's inequality
\begin{align*}
\E\left[\left(\frac{2}{n^{3/2}}\sup_{t\in[0,T]}\sum_{i=1}^n\int_0^tf'(\lambda^{(n)}_i)g_n(\lambda_i^{(n)})h_n(\lambda_i^{(n)})d\nu_i\right)^2\right]\leq \frac{4}{n^3}\sum_{i=1}^n\E\left[\int_0^T(f'(\lambda^{(n)}_i))^2g^2_n(\lambda_i^{(n)})h^2_n(\lambda_i^{(n)})ds\right].
\end{align*}
Since \eqref{eq:growth} gives $g^2_n(x)h^2_n(x)\leq 2K(1+|x|^2)$, we can apply \eqref{eq:pk:bound} from Lemma \ref{lem:pk:bound} to get
	\begin{align*}
	\E\left[\left(\frac{2}{n^{3/2}}\sup_{t\in[0,T]}\sum_{i=1}^n\int_0^tf'(\lambda^{(n)}_i)g_n(\lambda_i^{(n)})h_n(\lambda_i^{(n)})d\nu_i\right)^2\right]
	&\leq \|f'\|_{\infty}^2\frac{8K^2}{n^3}\sum_{i=1}^n\E\left[\int_0^T(1+(\lambda_i^{(n)})^2)ds\right]\\
	&\leq \|f'\|_{\infty}^2\frac{c_1}{n^2}\left(\frac{1}{n}\sum_{i=1}^n(\lambda_i^{(n)}(0))^2+1\right)\/,
	\end{align*}
	with $c_1=8K^2T(C_1(T,K,2)+1)$, where $C_1(T,K,2)$ is the constant appearing in \eqref{eq:pk:bound} for $k=2$.
By \eqref{eq:8momentbound} we can find a constant $K_T$ depending only on $T$ such that
\begin{align*}
\E\left[\left(\frac{2}{n^{3/2}}\sup_{t\in[0,T]}\sum_{i=1}^n\int_0^tf'(\lambda^{(n)}_i)g_n(\lambda_i^{(n)})h_n(\lambda_i^{(n)})d\nu_i\right)^2\right]\leq\|f'\|_{\infty}^2\frac{1}{n^2}K_T.
\end{align*}
Hence for any $ \varepsilon>0$ 
	\begin{eqnarray*}
	\sum_{n\geq 1}\pr\lefteqn{\left(\sup_{t\in[0,T]}\Bigg|\frac{2}{n^{3/2}}\int_0^tf'(\lambda^{(n)}_i)g_n(\lambda_i^{(n)})h_n(\lambda_i^{(n)})d\nu_i\Bigg|>\varepsilon\right)}\\
	&\leq&\frac{1}{\varepsilon^2}\sum_n\E\left[\left(\frac{2}{n^{3/2}}\sum_{i=1}^n\sup_{t\in[0,T]}\int_0^tf'(\lambda^{(n)}_i)g_n(\lambda_i^{(n)})h_n(\lambda_i^{(n)})d\nu_i\right)^2\right]
	\leq \frac{1}{\varepsilon^2}\sum_n\|f'\|_{\infty}^2\frac{1}{n^2}K_T<\infty.
	\end{eqnarray*}
Then the Borel-Cantelli Lemma implies that 
\begin{align*}
\lim_{n\to\infty}\frac{2}{n^{3/2}}\sup_{t\in[0,T]}\sum_{i=1}^n\int_0^tf'(\lambda^{(n)}_i)g_n(\lambda_i^{(n)})h_n(\lambda_i^{(n)})d\nu_i=0\qquad\text{a.s.}
\end{align*}
\end{proof}
Similar computations show that the forth component in \eqref{eq:EM:SDE}, which is non-zero only in the real-valued case, also vanishes when $n\to \infty$. 
\begin{proposition}
\label{prop:conv-sec-der}
For every fixed $T>0$ and $f\in\mathcal{C}_b^2(\R)$ we have 
\begin{equation}
\label{eq:term-sec-der}
\lim_{n\to\infty}\E\left|\frac{1}{n}\int_0^t\int_{\R}f''(x)G_n(x,x)\mu_s^{(n)}(dx)ds\right|=0
\end{equation}
for all $t\in[0,T]$.
\end{proposition}
\begin{proof}
Indeed, another application of \eqref{eq:pk:bound} gives
\begin{eqnarray*}
\E\left|\frac{1}{n}\int_0^t\int_{\R}f''(x)G_n(x,x)\mu_s^{(n)}(dx)ds\right| &=& \E\left|\frac{1}{n^2}\sum_{i=1}^n\int_0^tf''(\lambda_i^{(n)})G_n(\lambda_i^{(n)},\lambda_i^{(n)})ds\right|\\
&\leq& \frac{2K^2\|f''\|_{\infty}}{n^2}\sum_{i=1}^n\int_0^t\E(1+|\lambda_i^{(n)}|^2)ds\\
&\leq& \frac{C_T}{n}\|f''\|_{\infty}\left(1+\frac{1}{n}\sum_{i=1}^n|\lambda_i^{(n)}(0)|^2\right),
\end{eqnarray*}
where $C_T$ is a constant depending only on $T>0$. Hence, by \eqref{eq:8momentbound}, we can find a constant $C_{f,T}$ depending only on $f$ and $T$ such that
$$
\E\left|\frac{1}{n}\int_0^t\int_{\R}f''(x)G_n(x,y)\mu_s^{(n)}(dx)ds\right|\leq \frac{C_{f,T}}{n}.
$$
The previous inequality implies the result.
\end{proof}

\subsection{Tightness}
In this subsection we will prove the tightness of the family of laws $\{(\mu_t^{(n)})_{t\geq0}:n\in \N\}$ in the space $\mathcal{C}(\R_+,Pr(\R))$. The result is given in Proposition \ref{prop:tightness} below and it is a direct consequence of the following auxiliary result. 
\begin{lemma}
  \label{lem:tight:crit}
	For every fixed $T>0$ and $f\in \mathcal{C}_b^2(\R)$ there exists constant $C_2=C_2(K,T,f)$ such that
	\begin{equation}
	   \label{eq:bound_tightness}
	   \E\left(\langle\mu_t^{(n)},f\rangle-\langle\mu_s^{(n)},f\rangle\right)^4 \leq C_2\,(t-s)^2\/,
	\end{equation}
	for every $0\leq s<t\leq T$ and $n\in\N$.
\end{lemma}
\begin{proof}
  Using the semimartingale representation \eqref{eq:EM:SDE} we can write 
	\begin{equation*}
	   \E\left(\langle\mu_t^{(n)},f\rangle-\langle\mu_s^{(n)},f\rangle\right)^4 \leq 64 (M+D_1+D_2+D_3)\/,
	\end{equation*}
	where the difference between martingale parts is 
	\begin{eqnarray*}
	  M &=& \frac{16}{n^{6}}\,\E \left(\sum_{i=1}^n \int_s^t f'(\lni)g_n(\lni)h_n(\lni)d\nu_i\right)^4
	\end{eqnarray*}
	and the drift part can be naturally divided into three parts estimated as
	\begin{eqnarray*}
		D_1 &=& \frac{1}{n^4}\, \E\left( \int_s^t \int_{\R}f'(x)b_n(x)\mu_u^{(n)}(dx)du\right)^4\\
		D_2 &=& \frac{(2-\beta)^4}{16n^4}\E\left(\int_s^t \int_{\R}f''(x)G_n(x,x)\mu_u^{(n)}(dx)du\right)^4\\
		D_3 &=& \frac{\beta^4}{16}\E\left(\int_s^t \int_{\R^2} \frac{f'(x)-f'(y)}{x-y}\,G_n(x,y)\mu_{u}^{(n)}(dx)\mu_{u}^{(n)}(dy)du \right)^4\/.
	\end{eqnarray*}
	We will estimate each of those components separately. To deal with $M$ we apply the H\"older inequality and then the Burkholder-Davis-Gundy inequality to obtain
	\begin{eqnarray*}
	  \frac{1}{n^{6}}\,\E \left[\left(\sum_{i=1}^n \int_s^t f'(\lni)g_n(\lni)h_n(\lni)d\nu_i\right)^4\right]&\leq &\frac{1}{n^{3}}  \sum_{i=1}^n\E\left[\left( \int_s^t f'(\lni)g_n(\lni)h_n(\lni)d\nu_i\right)^4\right]\\
		&\leq &\frac{c_1}{n^3} \sum_{i=1}^n\E\left[\left( \int_s^t (f'(\lni))^2 g^2_n(\lni)h^2_n(\lni)du\right)^2\right]\/.
	\end{eqnarray*}
	Then, the boundedness of $f'$, the growth condition \eqref{eq:growth} and once again an application of H\"older inequality give us
	\begin{eqnarray*}
	  M \leq \frac{c_2}{n^3}  ||f'||_\infty^4 K^4 \sum_{i=1}^n\E\left( \int_s^t  (1+(\lni)^2)du\right)^2 \leq \frac{c_3(t-s)}{n^3}\sum_{i=1}^n \int_s^t \E(1+(\lni)^4)du \/,
	\end{eqnarray*}
	where $c_3=c_3(K,f)$.
	Finally, using \eqref{eq:pk:bound} we obtain 
	\begin{equation}
	 \label{eq:M:bound}
	  M \leq \frac{c_3(t-s)}{n^3}\int_s^t (n+ \E p_4^{(n)}(u))du \leq \frac{c_4}{n^2}(t-s)^2\left(1+\frac{1}{n} p_4^{(n)}(0)\right)\/, 
	\end{equation}
	with $c_4=c_4(K,T,f)$. On the other hand, in a similar way, we obtain
	\begin{equation*}
	  D_1 = \frac{1}{n^4}\E\left( \sum_{i=1}^n\int_s^t f'(\lni)\frac{b_n(\lni)}{n}du\right)^4 \leq \frac{(t-s)^3}{n}\E \sum_{i=1}^n\int_s^t \left(f'(\lni)\frac{b_n(\lni)}{n}\right)^4du.
	\end{equation*}	
	Using \eqref{eq:growth} together with \eqref{eq:pk:bound} lead to
	\begin{equation}
	\label{eq:D1:bound}
	 D_1 \leq {K^4}||f'||_\infty^4\frac{(t-s)^3}{n}\sum_{i=1}^n\int_s^t (1+|\lni|)^4\,du
		\leq C_{f,T}\,(t-s)^4\left(1+ \frac{1}{n}p_4^{(n)}(0)\right)\/.
	\end{equation}
	Since $G_n(x,x)\leq 2K^2(1+|x|)^2$ and $\beta\leq 2$, proceeding as previously, we arrive at
	\begin{equation}
	   \label{eq:D2:bound}
	   D_2 \leq \frac{(t-s)^3}{n}\, \E \sum_{i=1}^n\int_s^t \left(f''(\lni)G_n(x,x)\right)^4du \leq C_{f,T}\,(t-s)^4\left(1+ \frac{1}{n}p_8^{(n)}(0)\right)\/.
	\end{equation}
	Finally, we also have
	\begin{eqnarray*}
	  D_3 &=& \left(\frac{\beta}{2}\right)^4\,\E \left(\frac{1}{n^2}\ \sum_{i=1}^n\sum_{i=1}^n \int_s^t \frac{f'(\lni)-f'(\lnj)}{\lni-\lnj}\,G_n(\lni,\lnj)du\right)^4\\
		&\leq& \frac{(t-s)^3 ||f''||_{\infty}^{4}}{n^2} \sum_{i=1}^n\sum_{i=1}^n \E \left(\int_s^t G^4_n(\lni,\lnj)du\right)\/,
	\end{eqnarray*}
	which leads to existence of $c_4=c_4(K,f)>0$ such that 
	\begin{equation}
	   \label{eq:D3:bound}
		  D_3\leq c_4\,(t-s)^4\left(1+ \frac{1}{n}p_8^{(n)}(0)\right)\/.
	\end{equation}
	Collecting  \eqref{eq:M:bound}, \eqref{eq:D1:bound}, \eqref{eq:D2:bound} and \eqref{eq:D3:bound} we arrive at
	\begin{eqnarray*}
		 \E\left(\langle\mu_t^{(n)},f\rangle-\langle\mu_s^{(n)},f\rangle\right)^4 \leq c_5(t-s)^2\left(1+\frac{1}{n}p_4^{(n)}(0)+\frac{1}{n}p_8^{(n)}(0)\right)\/,
	\end{eqnarray*}	
The expression in the last bracket above is bounded in $n$ by \eqref{eq:8momentbound}, which ends the proof of \eqref{eq:bound_tightness}.
\end{proof}

	Therefore, by the well-known criterion (see Proposition 10.3 in \cite{bib:EthierKurtz1986}), the estimate \eqref{eq:bound_tightness} implies that the sequence of continuous real processes $\{(\langle\mu^{(n)}_t,f\rangle,t\geq0);n\geq1\}$ is tight and consequently we have the following result:
	
\begin{proposition}
  \label{prop:tightness}
	The family of measures $\{(\mu^{n}_t)_{t\geq 0}: n\in\N\}$ is tight in $\mathcal{C}(\R_+,\textrm{Pr}(\R))$.
\end{proposition}

\subsection{Characterization of the limits as the solution to an integral equation}
Let us assume that 
$f\in\mathcal{C}^2_c(\R)$ and let $r>0$ such that $\supp(f)\subseteq(-r,r)$. Recall that we now additionally assume that $g_n^2\to g^2$, $h_n^2\to h^2$ and $b_n/n\to b$ locally uniformly on $\R$. From Proposition \ref{prop:tightness} we know that the family $\{(\mu_{t}^{(n)})_{t\geq0}:n\geq1\}$ is relatively compact. Let us take a subsequence $\{(\mu_{t}^{(n_{k})})_{t\geq0}:k\geq1\}$ converging weakly to $(\mu_{t})_{t\geq0}$. Since \eqref{eq:EM:SDE} holds with $n$ replaced by $n_k$, this together with Propositions \ref{prop:conv_martin} and \ref{prop:conv-sec-der} imply
\begin{eqnarray*}
  \lim_{k\to\infty}\frac{2}{n_k^{3/2}}\sum_{i=1}^{n_k}\int_0^tf'(\lambda^{(n_k)}_i)g_{n_k}(\lambda_i^{(n_k)})h_{n_k}(\lambda_i^{(n_k)})d\nu_i&=&0,\quad a.s.\/,\\
	\lim_{k\to \infty}\frac{1}{n_k^2}\sum_{i=1}^{n_k}\int_0^tf''(x)G_{n_k}(x,x)\mu_s^{(n_k)}(dx)ds &=& 0\/,\quad a.s.\/,
\end{eqnarray*}
it is enough to deal with the third and fifth component in the RHS of \eqref{eq:EM:SDE}. First, we will show that 
\begin{equation*}
   \frac{1}{n_k}\int_0^t \int_{\R}f'(x)b_{n_k}(x)\mu_s^{(n_k)}(dx)ds \to \int_0^t \int_{\R}f'(x)b(x)\mu_s(dx)ds,
	\end{equation*}
	as $k\to \infty$. It follows from the fact that $f'(x)b_{n_k}(x)/n_k$ is bounded continuous function tending uniformly (on $\R$) to $f'(x)b(x)$, since $b_n/n$ is locally uniformly convergent and $f'$ has a compact support. This together with the weak convergence of $(\mu_t^{(n_k)})$ gives the result. 
	
		For the other term we note that we can write the expression
		\begin{equation*}
		 \int_0^t\int_{\R^2}\frac{f'(x)-f'(y)}{x-y}G_{n_k}(x,y)\mu_s^{(n_k)}(dx)\mu_s^{(n_k)}(dy)ds - \int_0^t\int_{\R^2}\frac{f'(x)-f'(y)}{x-y}G(x,y)\mu_s(dx)\mu_s(dy)ds
		\end{equation*}
		as the sum $A_1(n_k)+A_2(n_k)+B(n_k)$, where
		\begin{align*}
		 A_1(n_k) &=  \int_0^t\int_{|x|+|y|<R}\frac{f'(x)-f'(y)}{x-y}(G_{n_k}(x,y)-G(x,y))\mu_s^{(n_k)}(dx)\mu_s^{(n_k)}(dy)ds\/,\\
		A_2(n_k) &=  \int_0^t\int_{|x|+|y|\geq R}\frac{f'(x)-f'(y)}{x-y}(G_{n_k}(x,y)-G(x,y))\mu_s^{(n_k)}(dx)\mu_s^{(n_k)}(dy)ds\/,
		\end{align*}
		and $B(n_k)$ is equal to 
		\begin{equation*}
		 \int_0^t\int_{\R^2}\frac{f'(x)-f'(y)}{x-y}G(x,y)\mu_s^{(n_k)}(dx)\mu_s^{(n_k)}(dy)-\int_0^t\int_{\R^2}\frac{f'(x)-f'(y)}{x-y}G(x,y)\mu_s(dx)\mu_s(dy)ds\/.
		\end{equation*}
		Here $R>2r$. We use \eqref{eq:growth} to get
		\begin{equation*}
		|G_{n_k}(x,y)-G(x,y)|\leq 2K(1+|x|+|y|)(\sup_{x\leq R}\|g_{n_k}^2(x)-g^2(x)\|+\sup_{x\leq R}\|h_{n_k}^2(x)-h^2(x)\|)\/,
		\end{equation*}
		whenever $|x|+|y|\leq R$. Since $f\in\mathcal{C}_c^2(\R)$ we can find a constant $c_1=c_1(K,f)>0$ such that 
		\begin{equation*}
		 \sup_{|x|+|y|\leq R}\Bigg|\frac{f'(x)-f'(y)}{x-y}(G_{n_k}(x,y)-G(x,y))\Bigg|\leq c_1(\sup_{x\leq R}\|g_{n_k}^2(x)-g^2(x)\|+\sup_{x\leq R}\|h_{n_k}^2(x)-h^2(x)\|)
		\end{equation*}
		and consequently
		\begin{equation}
		  \label{eq:A1}
	A_1(n_k) \leq c_1t(\sup_{x\leq R}\|g_{n_k}^2(x)-g^2(x)\|+\sup_{x\leq R}\|h_{n_k}^2(x)-h^2(x)\|)\/.
		\end{equation}
		Using the fact that $\supp(f)\in(-r,r)$ together with the fact that $R>2r$ we obtain
	\begin{align*}
	c_2:=\sup_{x\in\R,|y|> R/2}\Bigg|\frac{f'(x)-f'(y)}{x-y}(G_{n_k}(x,y)-G(x,y))\Bigg|= \sup_{|x|\leq r, |y|> R/2}\Bigg|\frac{f'(x)}{x-y}(G_{n_k}(x,y)-G(x,y))\Bigg|<\infty.
	\end{align*}
	Therefore, by symmetry, we get
	\begin{equation}
	\label{eq:A2}
	A_2(n_k)\leq 2c_2 \int_0^t\sup_{k\in\N}\mu_s^{(n_k)}\left(|y|>R/2\right)ds.
	\end{equation}
	Collecting \eqref{eq:A1} and \eqref{eq:A2}  together with the fact that the sequence of measures $\{(\mu^{(n_k)}_t)_{t\geq 0}: k\in\N\}$ is tight we obtain
	\begin{align*}
	\limsup_{k\to\infty}|A_1(n_k)+A_2(n_k)|\leq 2c_2 \int_0^t\lim_{R\to\infty}\sup_{k\in\N}\mu_s^{(n_k)}\left(|y|>R/2\right)ds=0.
	\end{align*}
	Proceeding as previously, we can obtain
		\begin{align*}
	\sup_{(x,y)\in\R^2}\Bigg|\frac{f'(x)-f'(y)}{x-y}G(x,y)\Bigg|\leq c_1(\sup_{x\leq R}\|g^2(x)\|+\sup_{x\leq R}\|h^2(x)\|)+2\sup_{|x|\leq r, |y|> R/2}\Bigg|\frac{f'(x)}{x-y}G(x,y)\Bigg|<\infty.
	\end{align*}
	Hence, using the weak convergence of the sequence of measures $\{(\mu^{(n_k)}_t)_{t\geq 0}: k\in\N\}$
	to $(\mu_t)_{t\geq0}$, we obtain by the Dominated Convergence Theorem
	\begin{align*}
	\lim_{k\to\infty}B(n_k)=0.
	\end{align*}	
	These arguments imply that for any $t>0$
	\begin{eqnarray*}
	   \int_0^t\int_{\R^2}\frac{f'(x)-f'(y)}{x-y}G_{n_k}(x,y)\mu_s^{(n_k)}(dx)\mu_s^{(n_k)}(dy)ds \to \int_0^t\int_{\R^2}\frac{f'(x)-f'(y)}{x-y}G(x,y)\mu_s(dx)\mu_s(dy)ds\/.
	\end{eqnarray*}

Note that we can get rid of the additional assumption that $f$ has a compact support by approximation. However, we will generalize it to even broader class of functions $f$ (see Remark \ref{rem:classf}) in the next section. Finally, the last part of Theorem \ref{thm:main} is obvious. 


\section{Moments}
\label{sec:moments}

This section is devoted to study the properties of the moments of the empirical measure-valued processes and their convergence, i.e. we provide the proof of Theorem \ref{thm:moments}. We begin with its first part. 

\subsection{Finiteness of the moments} Following the assumptions of Theorem \ref{thm:moments}, we consider $(\mu_t)_{t\geq 0}$ as a limit in $\mathcal{C}(\R_+,\textrm{Pr}(\R))$ of a subsequence $(\mu_t^{(n_i)})_{t\geq 0}$, $i\in\N$. We assume that \eqref{eq:growth} holds for the continuous coefficients $g_n$, $h_n$, $b_n/n$ such that  $g_n^2$, $h_n^2$, $b_n/n$ are uniformly convergent. The second part of Theorem \ref{thm:main} implies that $(\mu_t)_{t\geq 0}$ solves the integral equation \eqref{limit}. 

First we show that \eqref{eq:momentbounds:0} implies \eqref{eq:momentbounds:t}, i.e. the uniform boundedness of $(2k)$-moments of $\mu_0^{(n_i)}$ gives the uniform boundedness of $(2k)$-moments of the limiting distributions $\mu_t$. Let us write $f_k(x)=|x|^k$ and consider $\varphi\in\mathcal{C}_c^2(\R)$ such that $\varphi\geq 0$. Then, for $0\leq t<T$ with fixed $T>0$, we have
\begin{eqnarray*}
   \langle\mu_t,\varphi\cdot f_{2k}\rangle &=& \lim_{i}\langle \mu_t^{(n_i)},\varphi\cdot f_{2k}\rangle \leq \limsup_{i}\langle \mu_t^{(n_i)}, f_{2k}\rangle \leq\limsup_i \frac{1}{n_i} \E p_{2k}^{(n_i)}(t)\\
	&\leq& \limsup_{i} C_{k,T}  \left(1+\frac{1}{n_i}p_{2k}^{(n_i)}(0)\right)\/,
\end{eqnarray*}
where $C_{k,T}$ does not depend on $i$. The last estimate is taking from Lemma \ref{lem:pk:bound}. Since the right-hand side does not depend on $\varphi$, we can conclude that
\begin{equation}
  \label{eq:moments:bound1}
  \langle \mu_t,f_{2k} \rangle \leq \tilde{C}_{k,T}  \/,\quad t\leq T\/,
\end{equation}
whenever \eqref{eq:momentbounds:0} holds. Note that, in particular, it implies that \eqref{limit} holds for every function $f\in\mathcal{C}^2(\R)$ such that $f$, $f'$ and $f''$ have sub-polynomial growth at infinity, i.e. there exists $k\geq 1$
\begin{equation}
\label{eq:f:subpolynomial}
   |f(x)|+|f'(x)|+|f''(x)|\leq K(1+|x|^k)\/,\quad x\in\R\/,
\end{equation}
 whenever \eqref{eq:8momentbound} holds.

\subsection{Convergence of the moments} Assuming \eqref{eq:momentbounds:0}, we fix $m\in\{0,\ldots,2k-1\}$ and write
\begin{equation}
\int_{\R}x^m \mu_t^{(n_i)}(dx) = \int_{|x|<\alpha}x^m\mu^{(n_i)}_t(dx)+\int_{|x|\geq\alpha}x^m\mu^{(n_i)}_t(dx)\/,
\end{equation}
where the first integral in the right-hand side is equal to 
\begin{equation*}
\int_0^{\alpha}\mu^{(n_i)}_t\{x\in\R:r<x^m<\alpha\}dr-\int_{-\alpha}^0\mu^{(n_i)}_t\{x\in\R:-\alpha<x^m<r\}dr\/.
\end{equation*}
Here $\alpha=\alpha(t)>0$ is chosen to have $\mu_t(\{\alpha\})=0$. Since the sequence $(\mu^{(n_i)}_t)_{i\geq1}$ converges weakly to $\mu_t$, the Lebesgue's Dominated Convergence Theorem gives
\begin{equation*}
\lim_{i\to\infty}\int_0^{\alpha}\mu^{(n_i)}_t\{x\in\R:r<x^m<\alpha\}dr = \int_0^{\alpha}\mu_t\{x\in\R:r<x^m<\alpha\}dr\/,
\end{equation*}
where we used the fact that the set $\{r\in[0,\alpha]: \mu_t(\{r\})>0\}$ has the Lebesgue measure zero. The same arguments give 
\begin{equation*}
\lim_{i\to\infty}\int_{-\alpha}^0\mu^{(n_i)}_t\{x\in\R:-\alpha<x^m<r\}dr = \int_{-\alpha}^0\mu_t\{x\in\R:-\alpha<x^m<r\}dr
\end{equation*}
and hence
\begin{equation}
\label{con_mom_1}
\lim_{i\to\infty}\int_{|x|<\alpha}x^m\mu^{(n_i)}_t(dx)=\int_{|x|<\alpha}x^m\mu_t(dx)\/.
\end{equation}
On the other hand, since
\begin{equation*}
\Bigg|\sup_{i}\int_{|x|\geq\alpha}x^m\mu^{(n_i)}_t(dx)\Bigg|\leq \frac{1}{\alpha^{2k-m}}\sup_{i}\int_{\R}x^{2k}\mu^{(n_i)}_t(dx)\/,
\end{equation*}
we can use \eqref{eq:momentbounds:t} to see that
\begin{equation}
\label{con_mom_2}
\lim_{\alpha\to\infty}\sup_{i}\int_{|x|\geq\alpha}x^m\mu^{(n_i)}_t(dx)=0\/.
\end{equation}
Therefore, using \eqref{con_mom_1} we obtain that
\begin{align*}
\limsup_{i\to\infty}\int_{\R} x^m \mu_t^{(n_i)}(dx) &=\limsup_{i\to\infty}\int_{|x|<\alpha}x^m\mu^{(n_i)}_t(dx)+\limsup_{i\to\infty}\int_{|x|\geq\alpha}x^m\mu^{(n_i)}_t(dx)\\
&=\int_{|x|<\alpha}x^m\mu_t(dx)+\limsup_{i\to\infty}\int_{|x|\geq\alpha}x^m\mu^{(n_i)}_t(dx).
\end{align*}
Finally, taking $\alpha\to\infty$, by using the Monotone Convergence Theorem together with \eqref{con_mom_2} we obtain
\begin{align*}
\limsup_{i\to\infty} \langle \mu_t^{(n_i)}, x^m\rangle=\int_{\R} x^m\mu_t(dx)\/.
\end{align*}
The similar computations give that 
\begin{equation*}
   \liminf_{i\to\infty}\int_{\R} x^m \mu_t^{(n_i)}(dx)=\int_{\R} x^m \mu_t(dx)\/,
\end{equation*}	
which ends the proof.

\subsection{Characterization by the moments} To prove the second part of Theorem \ref{thm:moments}, we assume that $(\mu_t)_{t\geq}$ is a solution to \eqref{limit}, where 
\begin{equation}
   \label{eq:moments:assump}
  g^2(x)+h^2(x)\leq K(1+|x|)\/,\quad x\in \R\/,\quad b(x)\leq K\/,\quad x\geq 0\/,
\end{equation}
and $b$ is non-negative on $(-\infty,0))$. We denote
\begin{equation}
   a_k(t) = \int_{\R} |x|^{k}\mu_t(dx)\/,\quad k=1,2,\ldots\/.
\end{equation} 
Since we consider $\mu_0$ which has an analytic characteristic function on a neighborhood of the origin, we have
\begin{equation}\label{D_bound}
   E: = \limsup_{k\in\N}\sqrt[k]{a_k(0)/k!}<\infty\/. 
\end{equation}
Using mathematical induction on $k$, we will show that 
\begin{equation}
\label{ref:akt:est}
   \sup_{t\leq T} a_k(t)\leq (k-1)!\, C^{2k-1} \/,
\end{equation}
which implies that the measure $\mu_t$ has characteristic functions being analytic a neighborhood of the origin and the measure is determined by its moments. Here
\begin{equation}
   \label{eq:CT}
  C = C(T,K) = \max\{1+\tilde{C}_{2,T},E,2+2KT+48T K^2\}\/,
\end{equation}
where $\tilde{C}_{2,T}$ is the constant appearing in \eqref{eq:moments:bound1}. Indeed, we have $a_0(t)=1$ and $a_1(t)\leq 1+a_2(t)\leq 1+\tilde{C}_{2,T}$ by \eqref{eq:moments:bound1}. Now assume that \eqref{ref:akt:est} holds for every $i\leq k$, where $k\geq 2$. Since $f_k$ is a function in $\mathcal{C}_2(\R)$ (for $k\geq 2$), such that \eqref{eq:f:subpolynomial} holds, we can use \eqref{limit} to write
\begin{eqnarray*}
   a_k(t) &=& \langle\mu_0, f_k\rangle +\int_{0}^{t}ds\int_{\R}b(x)f'_k(x)\mu_s(dx)+\frac{\beta}{2} \int_{0}^{t}ds\int_{\R^{2}}
	\frac{f^{\prime}_k(x)-f^{\prime}_k(y)}{x-y}G(x,y)\mu_{s}(dx)\mu_{s}(dy)\/.
\end{eqnarray*}
Since $f'_k(x)\leq 0$, $b(x)\geq 0$ on $(-\infty,0]$ and \eqref{eq:moments:assump} holds, we can use the inductive hypothesis to get
\begin{eqnarray*}
   \int_{\R}b(x)f'_k(x)\mu_s(dx) &\leq& K\, k\int_{(0,\infty)} x^{k-1}\mu_s(dx)
	\leq K\, k a_{k-1}(s) \leq K\, k (k-2)!C^{2k-3}\/.
\end{eqnarray*}
To deal with the next integral we use the following estimate
\begin{eqnarray*}
   \left|\frac{f_k'(x)-f_k'(y)}{x-y}\right|\leq k\sum_{i=0}^{k-2}|x|^{i}|y|^{k-2-i}\/,\quad x,y\in\R\/.
\end{eqnarray*}
For $k\in2\N$ or $xy\geq 0$ this is obvious. For $k\in 2\N+1$ and $xy<0$ it follows from 
\begin{eqnarray*}
   \left|\frac{f_k'(x)-f_k'(y)}{x-y}\right| &=& k\,\left|\frac{x^{k-1}+y^{k-1}}{x-y}\right|\leq k (|x|^{k-2}+|y|^{k-2})\/.
\end{eqnarray*}
Thus we can proceed in the following way
\begin{eqnarray*}
\int_{\R^{2}}\frac{f^{\prime}_k(x)-f^{\prime}_k(y)}{x-y}G(x,y)\mu_{s}(dx)\mu_{s}(dy) &\leq& 2k \sum_{i=0}^{k-2}\int_{\R}g^2(x)|x|^{i}\mu_s(dx)\int_{\R}h^2(x)|x|^{k-2-i}\mu_s(dx)\\
&\leq &2K^2\,k \sum_{i=0}^{k-2}(a_i(s)+a_{i+1}(s))(a_{k-2-i}(s)+a_{k-1-i}(s))\/,
\end{eqnarray*}
where the last inequality follows from \eqref{eq:moments:assump}. Since the right-hand side of \eqref{ref:akt:est} is increasing in $k$, 
we have 
\begin{equation*}
   a_i(s)+a_{i+1}(s)\leq 2i!C^{2i+1}\/,\qquad a_{k-2-i}(s)+a_{k-1-i}(s)\leq 2(k-2-i)!C^{2k-3-2i}\/,
\end{equation*}
and consequently the above-given integral over $\R^2$ is estimated from above by 
\begin{eqnarray*}
8K^2\,k\, \sum_{i=0}^{k-2}C^{2i+1}C^{2k-3-2i}i!(k-2-i)! \leq  8K^2 C^{2k-2}\,k(k-2)!\,\sum_{i=0}^{k-2}\frac{1}{\binom{k-2}{i}}\/.
\end{eqnarray*}
Taking into account the fact that $\sum_{i=0}^n1/\binom{n}{i}$ bounded by $3$ we finally obtain
\begin{eqnarray*}
\int_{\R^{2}}\frac{f^{\prime}_k(x)-f^{\prime}_k(y)}{x-y}G(x,y)\mu_{s}(dx)\mu_{s}(dy)
&\leq & 24K^2 \,k(k-2)!\,C^{2k-2}\leq 48K^2 (k-1)! C^{2k-2}\/.
\end{eqnarray*}
Collecting all together and using \eqref{eq:CT} we get
\begin{eqnarray*}
  a_k(t) &\leq& a_k(0)+2KT(k-1)!\,C^{2k-3}+48K^2T\, (k-1)!\, C^{2k-2} \leq (k-1)!\left(2+2KT+48T K^2\right)C^{2k-2}\\
				 &\leq& (k-1)!C^{2k-1}\/,
\end{eqnarray*}
for every $t\leq T$. Here we have used the fact that 
\begin{equation*}
 a_k(0)\leq E^k k!\leq  (k-1)! k C^k\leq 2(k-1)! C^{2k-2}\/,
\end{equation*}
where the last inequality holds since $C>2$. This ends the proof. 


\subsection{Cauchy transforms} We end this section by providing description of the Cauchy transforms of the families of laws $(\mu_t)_{t\geq 0}$ characterized as solutions to \eqref{limit}. This will be used in Section \ref{sec:Free} to establish the links between the distributions obtained as limits of general matrix flows and free diffusions. 

Let $\C^+=\{z\in\C: \Im(z)>0\}$ and denote
\begin{align*}
f_z(x):=\frac{1}{(x-z)},\qquad\text{for all $x\in\R, z\in\C^+$}.
\end{align*}
Using \eqref{limit} we obtain that
\begin{align}\label{ct1}
\frac{d}{dt}\langle f_z,\mu_t\rangle=-\int_{\R}\frac{b(x)}{(x-z)^2}\mu_t(dx)+\frac{1}{2}\int_{\R^{2}}
\frac{x-z+y-z}{(x-z)^2(y-z)^2}G(x,y)\mu_{t}(dx)\mu_{t}(dy).
\end{align}
Now noting that $G(x,y)=g^2(x)h^2(y)+g^2(y)h^2(x)$ for $x,y\in\R$ we get
\begin{align*}
\int_{\R^{2}}\frac{x-z+y-z}{(x-z)^2(y-z)^2}G(x,y)\mu_{t}(dx)\mu_{t}(dy)&=2\int_{\R}\frac{g^2(x)}{(x-z)}\mu_t(dx)\int_{\R}\frac{h^2(y)}{(y-z)^2}\mu_t(dy)\\&+2\int_{\R}\frac{g^2(x)}{(x-z)^2}\mu_t(dx)\int_{\R}\frac{h^2(y)}{(y-z)}\mu_t(dy).
\end{align*}
Hence, using the above computation in \eqref{ct1} 
\begin{align}\label{ct2}
\frac{d}{dt}\langle f_z,\mu_t\rangle&=-\int_{\R}\frac{b(x)}{(x-z)^2}\mu_t(dx)+\int_{\R}\frac{g^2(x)}{(x-z)}\mu_t(dx)\int_{\R}\frac{h^2(y)}{(y-z)^2}\mu_t(dy)
\notag\\&+\int_{\R}\frac{g^2(x)}{(x-z)^2}\mu_t(dx)\int_{\R}\frac{h^2(y)}{(y-z)}\mu_t(dy).
\end{align}	


\section{Wigner ensemble}
\label{sec:Wigner}
Using Theorem \ref{thm:main} we prove the following result, which can be seen as continuous time generalization of \textit{the classical Global semi-circular law}. The result proved in some special cases by Wigner in \cite{bib:Wigner:1958} and in the general form by Pastur  \cite{bib:Pastur:1973} states that for \textit{the Wigner matrice}s $M$ and a given interval $I$ one has
\begin{equation*}
    \lim_{n\to \infty} \frac{1}{n}N_I[\frac{1}{\sqrt{n}}M_n] = \int_I \rho^{sc}(x)dx\/.
\end{equation*}
Here $\rho^{sc}(x)=(2\pi)^{-1}\sqrt{4-x^2}\ind_{\{|x|\leq 2\}}$ is the density function of \textit{the semi-circular law} and $N_I[X]=\|\{1\leq i\leq n: \lambda_i(X)\in I\}\|$ is \textit{the eigenvalue counting function}. Recall that $n\times n$ Wigner Hermitian matrix is defined as a random Hermitian matrix $M_n=(m_{ij})_{1\leq i,j\leq n}$ such that $m_{ij}$ for $i< j$ are jointly independent and $m_{ji}=\overline{m_{ij}}$. Additionally it is required that $\ex m_{ij}=0$, $\ex|m_{ij}|^2=1$ for $i<j$ and $\ex m_{ii}=0$, $\ex m_{ii}^2=\sigma^2>0$. It is also sometimes assumed that $m_{ij}$, $i<j$ are identically distributed as well as $m_{ii}$. 

\begin{theorem}
\label{thm:Wigner1}
   Let $g_n$, $h_n$ and $b_n$ be continuous functions and $(X_t^{(n)})$ be a solution to 
	$$
	  dX_t^{(n)} = g_n(X_t^{(n)})dW_t^{(n)}h_n(X_t^{(n)})+h_n(X_t^{(n)})d(W_t^{(n)})^*g_n(X_t^{(n)}) +\frac{1}{n}b_n(X_t^{{n}})dt\/,\quad X_0^{(n)}\in\mathcal{H}_n\/,
	$$
	where $W_t^{(n)}= n^{-1/2}W_t$ with $W_t$ being a $n\times n$ complex matrix Brownian motion. If \eqref{eq:growth} holds, $g_n^2(x)\to 1/4$, $h_n^2(x)\to 1$ and $b_n(x)/n\to 0$, locally uniformly as $n\to \infty$, then the related empirical measure-valued process
	$$
	  \mu_t^{(n)}(dx) = \frac{1}{n}\sum_{i=1}^n \delta_{\lambda_i^{(n)}}(dx)
	$$
	tends in probability in the space $\mathcal{C}(\R_+,\textrm{Pr}(\R))$ to the family of semi-circular laws
	$$
	  \rho_t^{sc}(dx) = \frac{1}{2\pi t} \sqrt{4t-x^2}\ind_{\{|x|\leq 2\sqrt{t}\}}dx\/,
	$$
	whenever $\mu_0^{(n)}\Rightarrow \delta_0$ such that \eqref{eq:8momentbound} holds.
\end{theorem}

\begin{proof}
  In view of Theorem \ref{thm:main} it is enough to show that the equation
	\begin{equation}
	\label{eq:Wigner:mu}
	  	\langle\mu_{t},f\rangle=f(0)+ \frac{1}{2}\int_{0}^{t}ds\int_{\R^{2}}
	\frac{f^{\prime}(x)-f^{\prime}(y)}{x-y}\mu_{s}(dx)\mu_{s}(dy),
	\end{equation}
	has a unique solution $\rho_t^{sc}(dx)$. In fact it was done in \cite{bib:RogersShi:1993} by applying the Cauchy transform and showing that the obtained partial differential equation has a unique solution. However, we can get the result by looking at the moments of solutions as was done by Chan in \cite{bib:Chan:1993}. Note that the second part of Theorem \ref{thm:moments} gives that any solution to \eqref{eq:Wigner:mu} is uniquely determined by its moments. Moreover, it is easy to see that 
	$$
	   m_0(t) = \mu_t(\R) = 1\/,\quad m_1(t) = \int_{\R} x\mu_t(dx) = 0
	$$
	and for every $k\geq 2$ we have
	\begin{eqnarray*}
	   m_k(t):= \int_{R} x^k \mu_t(dx) &=& \frac{k}{2}\int_{0}^{t}ds\int_{\R^{2}}\sum_{i=0}^{k-2}x^{i}y^{k-2-i}\mu_{s}(dx)\mu_{s}(dy)
		= \frac{k}{2}\sum_{i=0}^{k-2}\int_0^t m_i(s)m_{k-2-i}(s)ds\/.
	\end{eqnarray*}
	First, it shows that all the solutions have the same moments, since $m_k(t)$ is given in terms of $m_i(t)$, $i\leq k-2$ and $m_0(t)=1$, $m_1(t)=0$. Additionally, by simple induction, one can prove that 
	$$
	  m_{2n+1}(t)=0\/,\quad m_{2n}(t) = \frac{t^n}{n+1}\binom{2n}{n}\/,\quad t\geq 0\/,
	$$
	and it agrees with the moments of the centered semi-circular law with variance $t$. The result follows.
\end{proof}
\begin{remark}
  Note that \eqref{eq:8momentbound} holds, for example, if the distributions $\mu_0^{(n)}(dx)$ have uniformly bounded supports.
\end{remark}

We also provide the real-valued analogue of the result
\begin{theorem}
\label{thm:Wigner2}
   Let $g_n$, $h_n$ and $b_n$ be continuous functions and $(X_t^{(n)})$ be a solution to 
	$$
	  dX_t^{(n)} = g_n(X_t^{(n)})dB_t^{(n)}h_n(X_t^{(n)})+h_n(X_t^{(n)})d(B_t^{(n)})^Tg_n(X_t^{(n)}) +\frac{1}{n}b_n(X_t^{{n}})dt\/,\quad X_0^{(n)}\in\mathcal{S}_n\/,
	$$
	where $B_t^{(n)}= n^{-1/2}B_t$ with $B_t$ being a $n\times n$ real-valued matrix Brownian motion. If \eqref{eq:growth} holds, $g_n^2(x)\to 1/4$, $h_n^2(x)\to 1$ and $b_n(x)/n\to 0$, locally uniformly as $n\to \infty$, then the related empirical measure-valued process
	$$
	  \mu_t^{(n)}(dx) = \frac{1}{n}\sum_{i=1}^n \delta_{\lambda_i^{(n)}}(dx)
	$$
	tends in probability in the space $\mathcal{C}(\R_+,\textrm{Pr}(\R))$ to the family of semi-circular laws
	$$
	  \rho_t^{sc}(dx) = \frac{1}{\pi t} \sqrt{2t-x^2}\ind_{\{|x|\leq \sqrt{2t}\}}dx\/,
	$$
	whenever $\mu_0^{(n)}\Rightarrow \delta_0$ such that \eqref{eq:8momentbound} holds.
\end{theorem}
\begin{proof}
   The result follows from the observation that the simple time scaling $\mu_{t}(dx)=\nu_{2t}(dx)$ transforms the equation
		$$
	  	\langle\nu_{t},f\rangle=f(0)+ \frac{1}{4}\int_{0}^{t}ds\int_{\R^{2}}
	\frac{f^{\prime}(x)-f^{\prime}(y)}{x-y}\nu_{s}(dx)\nu_{s}(dy),
	$$
	into \eqref{eq:Wigner:mu}.
\end{proof}

\begin{remark}
  Note that in both Theorem \ref{thm:Wigner1} and Theorem \ref{thm:Wigner2}, apart from the trivial case $g_n=1$, $h_n=1/2$ and $b_n=0$, the assumption that the entries of $X^{(n)}_t$ are independent (above and on the diagonal) does not hold. It is replaced by the requirement that $X^{(n)}$ is a solution to given matrix-valued SDE, which imposes certain structure of the considered matrices. 
\end{remark}

\begin{remark}
\label{rem:Wigner:diffinit}
 Other generalizations of Theorem \ref{thm:Wigner1} and Theorem \ref{thm:Wigner2} follow by considering different limits $\mu_0(dx)$ of the initial distributions $\mu_t^{(n)}(dx)$. Assuming that the measure $\mu_0(dx)$ fulfills the assumptions of the last part of Theorem \ref{thm:moments}, we can easily show that the uniqueness of a solution to \eqref{eq:Wigner:mu} (with $f(0)=\left<\delta_0,f\right>$ replaced by $\left<\mu_0,f\right>$) holds and the convergence of the empirical measure-valued processes follows. However, the limiting family of measures $\mu_t(dx)$ is then no longer given by the semi-circle law, but still it can be determined by its moments given by the corresponding recursive differential equations. 
\end{remark}


\section{Generalized Wishart and Laguerre ensemble}
\label{sec:Wishart}
Another classical result on the convergence of spectral empirical measure relates to "squared" matrices and the Marchenko-Pastur distribution \cite{bib:MarchenkoPastur:1967}. The time-continuous analogue refers to the classical Laguerre stochastic differential equation 
	\begin{eqnarray*}
	  dX_t = \sqrt{X_t}dW_t+dW_t^*\sqrt{X_t}+mI_ndt\/,\quad X_t\in \mathcal{H}_n\/,
	\end{eqnarray*}
   where $m\in\N$ and $I_n$ stands for the identity matrix in $\mathcal{H}_n$. It is well-known that $X_t = N_t\times N_t^T$ is a solution to above given SDE, where $N_t$ is $n\times m$ complex-valued Brownian matrix. The Laguerre process was studied e.a. by Konig and O'Connel in \cite{bib:KonigOConnel:2001} and Demni in \cite{bib:Demni:2007}. This section is devoted to study in details the following generalization 
	\begin{equation}
	\label{eq:Wishart:general}
	 dX_t = \sqrt{|X_t|}dW_t+dW_t^T \sqrt{|X_t|}+\alpha I_n dt\/,\quad X_t\in \mathcal{H}_n\/,
	\end{equation}
 with $\alpha\in\R$. The appearance of the norms is forced by the fact that a solution does not belong to $\overline{\mathcal{S}_n^{+}}$ for $\alpha<n-1$ and $\alpha\notin\{1,2,\ldots,n-2\}$ (see \cite{bib:gmm2018} for the study of the real-valued case). The real-valued version \eqref{eq:Wishart:general} with $W_t$ replaced by $B_t$, $\mathcal{H}_n$ by $\mathcal{S}_n$ and more general SDEs with $g_n^2(x)\to |x|$, $h_n(x)\to1$, and $b_n(x)/n\to \alpha$ will also be considered. The solutions in this case are known as Wishart processes and were introduced by Bru in \cite{bib:b91}.

Recall that the corresponding integral equations for limits of subsequences of the scaled empirical measures family is 
		\begin{equation} 
		\label{eq:Wishart:integral}
	\langle\mu_{t},f\rangle=\langle \mu_0,f\rangle+\alpha \int_{0}^{t}ds\int_{\R}f'(x)\mu_s(dx)+\frac{\beta}{2} \int_{0}^{t}ds\int_{\R^{2}}
	\frac{f^{\prime}(x)-f^{\prime}(y)}{x-y}(|x|+|y|)\mu_{s}(dx)\mu_{s}(dy),
	\end{equation}
	with $f\in\C^2(\R)$ such that $f,f',f''$ have sub-polynomial growth at infinity. In view of Theorem \ref{thm:main}, the convergence of $\mu_t^{(n)}(dx)$ is related to the uniqueness of solutions to \eqref{eq:Wishart:integral}. As we will see below, it does not hold in full generality. First, recall the following well-known result. 
	
	\medskip
	
	\begin{proposition}
	\label{prop:Wishart:Unique}
	  For every $\alpha\geq 0$ and $\mu_0(dx)=\delta_0(dx)$ there exists a unique solution $(\mu_t)_{t\geq 0}$ of \eqref{eq:Wishart:integral} such that $\mu_t$ is supported in $[0,\infty)$ for every $t\geq 0$. 
	\end{proposition}
	
	The result was shown in \cite{bib:CabanalGuionnet:2001} by applying the Cauchy transform. The solution (for $\beta=2$) is given by the family of dilations of the Marchenko-Pastur distribution
	\begin{equation*}
	   \nu_t^{MP}(dx) = (1-\alpha)_+\,\delta_0(dx)+\frac{\sqrt{(x-at)(bt-x)}}{2\pi xt}\ind_{(ta,tb)}(x)dx\/,
	\end{equation*}
	where $a=(1-\sqrt{\alpha})^2$ and $b=(1+\sqrt{\alpha})^2$. Note the presence of the atom at zero for $\alpha<1$. However, one can also deduce the following recursive integral equations for the moments $m_k(t) = \int_\R x^k \mu_t(dx)$ of a solution $\mu_t(dx)$ to  \eqref{eq:Wishart:integral}
	\begin{eqnarray*}
	  m_0(t) &=& 1\/,\\
	  m_k(t) &=& \alpha k \int_0^t m_{k-1}(s)ds+\beta k \sum_{i=0}^{k-2}\int_0^t m_{i+1}(s)m_{k-2-i}(s)ds\/,\quad k=1,2,\ldots 
	\end{eqnarray*}
	and show that it has the unique solution coinciding with the moments of $\nu_t^{MP}(dx)$. Note that the real-valued case ($\beta=1$) can be deduced from the complex-valued case  by simple time-scaling, similarly as it was done in Theorem \ref{thm:Wigner2}. Note that it transforms solution for $\alpha\geq 0$ and $\beta=2$ into $\beta=1$ and $\alpha/2\geq 0$.
	
	\medskip
	
	Let us denote the unique solution to \eqref{eq:Wishart:integral}, with $\alpha\geq 0$, supported in $[0,\infty)$ for every $t\geq 0$ by $\mathcal{MP}^+(\alpha)$. Here we assume that $\mu_0(dx)=\delta_0(dx)$. Analogously, we write $\mathcal{MP}^{-}(\alpha)$ for $\alpha\leq 0$ for a solution with support in $(-\infty,0]$. It is easy to check that the later are just reflections of $\mathcal{MP}^+(\alpha)$ into negative half-line, i.e. 
	$$
	  \mathcal{MP}^-(-\alpha) = \{(\mu_t(dx))_{t\geq 0}: \mu_t(A)=\nu_t(-A), (\nu_t)_{t\geq 0}\in \mathcal{MP}^{+}(\alpha)\}\/,\quad \alpha\geq 0\/.
	$$
	Let us also write $\mathcal{MP}(\alpha)$ for a set of solutions to \eqref{eq:Wishart:integral} without restriction on the support of $\mu_t$. This distinction is made since $\mathcal{MP}(\alpha)\neq \mathcal{MP}^+(\alpha)$ for $\alpha\in[0,1)$, i.e. the uniqueness of the solution to \eqref{eq:Wishart:integral} does not hold for $\alpha\in[0,1)$. This is an immediate consequence of Proposition \ref{prop:Wishart:Unique} and the following result.
	 
	\begin{proposition}
	\label{prop:Wishart:plus}
For every $\alpha\in[0,1)$ and $\mu_0(dx)=\delta_0(dx)$ there exists a solution $(\mu_t)_{t\geq 0}$ such that $\mu_t(-\infty,0)>0$ for very $t>0$.
	\end{proposition}
	
	\begin{proof}
For given $\alpha\in[0,1)$ we define $\lambda = (1+\alpha)/2\in[1/2,1)$ and $\lambda^* = (1-\alpha)/2\in(0,1/2]$. Let $\nu_t^+(dx)\in\mathcal{MP}^+(1)$, be the unique solution to
\begin{eqnarray*}
   \left<{\nu}_t^+,f\right>
	= f(0)+\int_0^t ds\int_{[0,\infty)}f'(x)\nu_s^+(dx)+\frac{1}{2}\int_{0}^{t}ds\int_{[0,\infty)^{2}}
	\frac{f^{\prime}(x)-f^{\prime}(y)}{x-y}(x+y)\nu_{s}^+(dx)\nu_{s}^+(dy)\/.
\end{eqnarray*}
Let $\nu_t^{-}(dx)\in\mathcal{MP}^{-}(1)$, be the reflection of $\nu_t^+(dx)$. We define
\begin{equation}
  \label{eq:Wishart:neg}
  \mu_t(dx) = \lambda \nu_{\lambda t}^+(dx)+\lambda^*\nu_{t\lambda^*}^-(dx)\/.
\end{equation}
Since $\lambda^*>0$, it is obvious that $\mu_t(-\infty,0)>0$ for every $t>0$. To see that $(\mu_t)_{t\geq 0}$ solves \eqref{eq:Wishart:integral} observe that 
\begin{eqnarray*}
      \left<{\nu}_{ \lambda t}^+,f\right>
	= f(0)+\lambda\int_0^t ds\int_{[0,\infty)}f'(x)\nu_{\lambda s}^+(dx)+\frac{\lambda}{2}\int_{0}^{t}ds\int_{[0,\infty)^{2}}
	\frac{f^{\prime}(x)-f^{\prime}(y)}{x-y}(x+y)\nu_{\lambda s}^+(dx)\nu_{\lambda s}^+(dy)\/.
\end{eqnarray*}
Moreover, denoting $\overline{f}(x)=f(-x)$, we get $ \left<\nu_{(1-\lambda) t}^-,f\right>=\left<{\nu}_{(1-\lambda) t}^+,\bar{f}\right>$, which is equal to 
\begin{eqnarray*}
   && f(0)+\lambda^*\int_0^t ds\int_{(0,\infty)}\bar{f}'(x)\nu_{\lambda^* s}^+(dx)+\frac{\lambda^*}{2} \int_{0}^{t}ds\int_{(0,\infty)^{2}}
	\frac{\bar{f}^{\prime}(x)-\bar{f}^{\prime}(y)}{x-y}(x+y)\nu_{\lambda^* s}^+(dx)\nu_{\lambda^* s}^+(dy),\\
	&=& f(0)-\lambda^*\int_0^t ds\int_{(-\infty,0)}f'(x){\nu}_{\lambda^* s}^-(dx)+\frac{\lambda^*}{2}\int_{0}^{t}ds\int_{(-\infty,0)^{2}}
	\frac{f^{\prime}(x)-f^{\prime}(y)}{x-y}(|x|+|y|){\nu}_{\lambda^* s}^-(dx){\nu}_{\lambda^* s}^-(dy).
\end{eqnarray*}
Here we used the fact that ${\nu}_t^+(\{0\})=\nu_t^-(\{0\})=0$. Since 
\begin{eqnarray}
   \label{eq:Wishart:plusminus}
   \int_{(-\infty,0)}\int_{(0,\infty)} \frac{f'(x)-f'(y)}{x-y}(|x|+|y|)\nu_{\lambda s}^+(dx){\nu}^-_{\lambda^* s}(dy) = \int_{(0,\infty)}f'(x)\nu_{\lambda s}^+(dx)-\int_{(-\infty,0)}f'(x){\nu}_{\lambda^* s}^-(dx)
\end{eqnarray}
we can write
\begin{eqnarray*}
\int_0^tds \int_{\R^2} \frac{f'(x)-f'(y)}{x-y}(|x|+|y|)\mu_{s}(dx)\mu_{s}(dy)
\end{eqnarray*}
as the following sum 
\begin{eqnarray*} 
&&\lambda^2\int_{0}^{t}ds\int_{[0,\infty)^{2}}
	\frac{f^{\prime}(x)-f^{\prime}(y)}{x-y}(x+y)\nu_{\lambda s}^+(dx)\nu_{\lambda s}^+(dy)+2\lambda\lambda^*\int_0^t ds\int_{(0,\infty)}f'(x)\nu_{\lambda s}^+(dx)\\
	&&-2\lambda\lambda^*\int_0^tds\int_{(-\infty,0)}f'(x){\nu}_{\lambda^*s }^-(dx)+(\lambda^*)^2\int_{0}^{t}ds\int_{(-\infty,0)^{2}}
	\frac{f^{\prime}(x)-f^{\prime}(y)}{x-y}(|x|+|y|){\nu}_{\lambda^*s}^-(dx){\nu}_{\lambda^*s}^-(dy)\/.
\end{eqnarray*}
Thus we get
\begin{eqnarray*}
   \langle \mu_t,f\rangle &=& \langle \mu_0,f\rangle +\alpha \int_0^tds\int_{\R}f'(x)\mu_t(dx)+\frac12\int_0^t \int_{\R^2} \frac{f'(x)-f'(y)}{x-y}(|x|+|y|)\mu_s(dx)\mu_s(dy)\/,
\end{eqnarray*}
where we used the fact that $2\lambda-1=\alpha$.

The solution constructed above is not the only one with non-positive support. Indeed, for given numbers $\alpha^-\geq \alpha^+> 1$ we define the following strictly positive constants
\begin{eqnarray*}
   \lambda = \frac{\alpha^--1}{\alpha^+\alpha^--1}\/,\quad \lambda^* = \frac{\alpha^+-1}{\alpha^+\alpha^--1}\/,\quad \gamma = \frac{(\alpha^+-1)(\alpha^--1)}{\alpha^+\alpha^--1}\/.
\end{eqnarray*}	
Note that $\lambda+\lambda^*+\gamma=1$. Let $\nu^+_t(dx)$ be $\mathcal{MP}^+(\alpha^+)$ and let $\nu^{-}_t(dx)$ be $\mathcal{MP}^-(-\alpha^-)$, i.e. the reflection of $\mathcal{MP}(\alpha^-)$ onto negative half-line. Note that $\nu^+_t(\{0\})=\nu^{-}_t(\{0\})=0$ since $\alpha^+$ and $\alpha^-$ are greater than $1$. We define
\begin{equation}
   \label{eq:Wishart:non2}
   \mu_t(dx) = \lambda \nu^{+}_{\lambda t}(dx)+ \gamma\delta_0(dx)+\lambda^*\nu^-_{\lambda^* t}(dx)\/,\quad t\geq 0\/.
\end{equation}
It is easy to see that 
\begin{eqnarray*}
   \langle \mu_t,f \rangle &=& \langle \mu_0,f \rangle + \alpha^+\lambda^2 \int_0^t ds\int_{(0,\infty)}f'(x)\nu^{+}_{\lambda s}(dx)-\alpha^-(\lambda^*)^2\int_0^t ds\int_{(-\infty,0)}f'(x)\nu^{-}_{\lambda^* s}(dx)\\
	&&+\frac{\lambda^2}{2}\int_0^tds \int_{(0,\infty)^2}\frac{f'(x)-f'(y)}{x-y}(x+y)\nu^{+}_{\lambda s}(dx)\nu^{+}_{\lambda s}(dy)\\
	&&+\frac{(\lambda^*)^2}{2}\int_0^tds \int_{(-\infty,0)^2}\frac{f'(x)-f'(y)}{x-y}(|x|+|y|)\nu^{-}_{\lambda^* s}(dx)\nu^{-}_{\lambda^* s}(dy)\/.
\end{eqnarray*}
However, the sum of the last two expressions is equal to the sum of the following three integrals
\begin{eqnarray*}
		&&\frac{1}{2}\int_0^tds \int_{\R^2} \frac{f'(x)-f'(y)}{x-y}(|x|+|y|)\mu_{s}(dx)\mu_{s}(dy)\\
		&&-\gamma \int_0^tds\int_{\R} \frac{(f'(x)-f'(0))|x|}{x}(\lambda \nu^+_{s\lambda}(dx)+\lambda^*\nu^-_{s\lambda^*}(dx))\\
		&& - \lambda\lambda^*\int_{(-\infty,0)}\int_{(0,\infty)} \frac{f'(x)-f'(y)}{x-y}(|x|+|y|)\nu^{+}_{\lambda s}(dx)\nu^-_{\lambda^* s}(dy) \/.
	\end{eqnarray*}
The last integral can be evaluated as in \eqref{eq:Wishart:plusminus}. The second integral is just equal to 
$$
   -\gamma\lambda \int_0^t ds \int_{(0,\infty)} f'(x)\nu^+_{s\lambda}(dx)+\gamma\lambda^* \int_0^t ds \int_{(-\infty,0)} f'(x)\nu^-_{s\lambda^*}(dx)+\gamma(\lambda-\lambda^*)f'(0)\/.
$$
Combining all together and using the fact that 
$$
  \lambda\alpha^+-\lambda^*-\beta = \lambda-\lambda^*\alpha^-+\beta = \lambda-\lambda^* = \frac{\alpha^--\alpha^+}{\alpha^+\alpha^--1}
$$
we arrive at
\begin{eqnarray*}
   \langle \mu_t,f\rangle &=& \langle \mu_0,f\rangle +\alpha \int_0^tds\int_{\R}f'(x)\mu_t(dx)+\frac12\int_0^t \int_{\R^2} \frac{f'(x)-f'(y)}{x-y}(|x|+|y|)\mu_s(dx)\mu_s(dy)
\end{eqnarray*}	
with 
\begin{eqnarray}
\label{eq:Wishart:alpha+-}
  \alpha = \frac{\alpha^--\alpha^+}{\alpha^+\alpha^--1}\/.
\end{eqnarray}
Note that for any given $\alpha\in[0,1)$ we can find $\alpha^-\geq \alpha^+>1$ such that \eqref{eq:Wishart:alpha+-}	holds. Moreover, for any $\alpha^-\geq \alpha^+>1$ the equality \eqref{eq:Wishart:alpha+-} defines a number $\alpha\in[0,1)$. 
	\end{proof}	
	
\begin{remark}
 In general, the fact that there are many solutions to \eqref{eq:Wishart:integral} does not imply that the weak limit of $\mu_t^{(n)}(dx)$ does not exist, in the case of the generalized Wishart SDE \eqref{eq:Wishart:general} one can simply construct a sequence of solutions $X_t^{(n)}$ such that $\mu_t^{(n)}$ converges to a solution with non-positive support. Indeed, using the results from \cite{bib:gm2019}, we can consider a solution $X_t^{(n)}$ to \eqref{eq:Wishart:general} with $\alpha_n$ such that $\alpha_n/n\to \alpha\in[0,1)$ and the initial distribution such that 
$$
\lambda_1^{(n)}(0)<\ldots<\lambda_{k^*}^{(n)}(0)<0<\lambda_{k^*+1}^{(n)}(0)<\ldots<\lambda_n^{(n)}(0)\/,
$$
where $k^*=\lceil \frac{n+1-\alpha_n}{2}\rceil$. Note that $k^*/n\to \frac{1-\alpha}{2}=\lambda^*$ and $1-k^*/n\to \frac{1+\alpha}{2}=\lambda$, where $\lambda$ and $\lambda^*$ are the constants in the first example constructed above. One can easily show that the sequence of the corresponding empirical measure-valued processes $\mu_t^{(n)}(dx)$ converges to \eqref{eq:Wishart:neg} by considering separately the positive and the negative part of $\mu_t^{(n)}(dx)$. The details as well as the corresponding construction for \eqref{eq:Wishart:non2} are left to the reader. 
\end{remark}	
	
\begin{remark}
   The above-constructed examples show also that the problem of the existence of the weak limit of $(\mu_t^{(n)})_{t\geq 0}$ is not related to the uniqueness of the solution to \eqref{eq:MSDE:hermitian}. 
\end{remark}	

\begin{theorem}
  \label{thm:Wishart:class}
   Let $g_n$, $h_n$ and $b_n$ be continuous functions and $(X_t^{(n)})$ be a solution to 
	$$
	  dX_t^{(n)} = g_n(X_t^{(n)})dW_t^{(n)}h_n(X_t^{(n)})+h_n(X_t^{(n)})d(W_t^{(n)})^*g_n(X_t^{(n)}) +\frac{1}{n}b_n(X_t^{{n}})dt\/,\quad X_0^{(n)}\in\mathcal{H}_n^+\/,
	$$
	where $W_t^{(n)}= n^{-1/2}W_t$ with $W_t$ being a $n\times n$ complex matrix Brownian motion and $\supp(\mu_0^{(n)})\subset[0,\infty]$. Assume that \eqref{eq:growth} holds, $g_n^2(x)\to |x|$, $h_n^2(x)\to 1$ and $b_n(x)/n\to \alpha\geq 0$, locally uniformly as $n\to \infty$. If additionally there exist $c_1(n), c_2(n), c_3(n)>0$ such that for every $x\geq 0$ we have
	\begin{equation}
	   \label{eq:Wishart:bounds}
	   g_n^2(x) \leq c_1(n)x\/,\quad
		 h_n^2(x) \leq  c_2(n)\/,\quad
		 b_n(x) \geq c_3(n) 
	\end{equation}
	and $c_3(n)\geq c_1(n)c_2(n)(\beta(n-1)+2)$ holds, then the related empirical measure-valued process
	$$
	  \mu_t^{(n)}(dx) = \frac{1}{n}\sum_{i=1}^n \delta_{\lambda_i^{(n)}}(dx)
	$$
	tends in probability in the space $\mathcal{C}(\R_+,\textrm{Pr}(\R))$ to the family of the Marchenko-Pastur distributions
	$$
	  \nu_t^{MP}(dx) = (1-\alpha)_+\,\delta_0(dx)+\frac{\sqrt{(x-at)(bt-x)}}{2\pi xt}\ind_{(ta,tb)}(x)dx\/,
	$$
	with $a=(1-\sqrt{\alpha})^2$ and $b=(1+\sqrt{\alpha})^2$, whenever $\mu_0^{(n)}\Rightarrow \delta_0$ such that \eqref{eq:8momentbound} holds.
\end{theorem}
	\begin{proof}
	   Recalling the result from Theorem \ref{thm:main}, the above-given discussion and Proposition \ref{prop:Wishart:plus}, the result will follow from the fact, that the additional assumptions \eqref{eq:Wishart:bounds} ensure that the solutions $X_t^{(n)}$ stay in the cone $\mathcal{H}_n^+$ of the positive-defined Hermitian matrices for every $t> 0$. To see this define $V_n(t)=\log e_n(t)$, where $e_n(t) = \lambda_1^{(n)}(t)\cdot\ldots\cdot\lambda_n^{(n)}(t)$ is the determinant of $X^{(n)}_t$, being well-defined for $t<T_0^{(n)}$, where $T_0^{(n)}=\inf\{t\geq 0: \lambda_1^{(n)}(t)=0\}$. Applying \eqref{eq:sympol:SDE} and the Ito's formula we obtain
			\begin{eqnarray}
			  \label{eq:Vn:SDE}
			   dV_n &=& \frac{2}{\sqrt{n}e_n}\sum_{i=1}^n g_n(\lambda_i^{(n)})h_n(\lambda_i^{(n)})e_{n-1}^{\overline{\lambda_i^{(n)}}}d\nu_i-\frac{2}{n\,e_n^2}\sum_{i=1}^n g_n^2(\lambda_i^{(n)})h_n^2(\lambda_i^{(n)})(e_{n-1}^{\overline{\lambda_i^{(n)}}})^2\,dt\\
				&&+\frac{1}{e_n}\left(\sum_{i=1}^n \frac{b_n(\lambda_i^{(n)})}{n}e_{n-1}^{\overline{\lambda_i^{(n)}}}-\frac{\beta}{n}\sum_{i<j}G_n(\lambda_i^{(n)},\lambda_j^{(n)})e_{n-2}^{\overline{\lambda_i^{(n)}},\overline{\lambda_j^{(n)}}}\right)dt\/, \quad t<T_0^{(n)}\/,\nonumber
			\end{eqnarray}
			where $\nu_1,\dots,\nu_n$ are independent Brownian motions. Using \eqref{eq:Wishart:bounds} and the relations
			\begin{eqnarray*}
			   \lambda_i^{(n)}e_{n-1}^{\overline{\lambda_i^{(n)}}} = e_n\/,\quad \sum_{i=1}^ne_{n-1}^{\overline{\lambda_i^{(n)}}}=e_{n-1}\/,\quad \sum_{i<j}(\lambda_i^{(n)}+\lambda_j^{(n)})e_{n-2}^{\overline{\lambda_i^{(n)}},\overline{\lambda_j^{(n)}}} = (n-1)e_{n-1}
			\end{eqnarray*}
			we get the following bounds on the drift part of $V_n$
			\begin{eqnarray*}
			   \textrm{drift}[V_n]\geq \frac{e_{n-1}}{n\,e_n}(c_3(n)-c_1(n)c_2(n)[\beta(n-1)+2])\geq 0
			\end{eqnarray*}
			for every $t<T_0^{(n)}$. Since the finite-variation part $\textrm{drift}[V_n]$ is bounded from below, it can not explode to $-\infty$ in finite time. Thus, using the McKean's argument we conclude that $T_0^{(n)}=\infty$ a.s. (for every $n$) and consequently $\lambda_1^{(n)}(t)>0$ for every $t>0$. It means that $\supp(\mu_t^{(n)})\in[0,\infty)$ and the result follows from Theorem \ref{thm:main} and Proposition \ref{prop:Wishart:plus}.
	\end{proof}
	\begin{remark}
	   As in the previous section, we can easily obtain the analogous result for the real-valued case, by changing $\beta=2$ to $\beta=1$. We can generalized both results by taking more general limits of the initial distributions $\mu_0^{(n)}$ (as it is described in Remark \ref{rem:Wigner:diffinit}).
	\end{remark}
	

\section{Geometric Matrix Brownian motion}
\label{sec:GMBM}

We will consider the solutions to the following sequence of SDEs,
\begin{align}
\label{eq:GMBM:SDE}
dX^{(n)}_t=\sqrt{X_t^{(n)}}dW^{(n)}_t\sqrt{X_t^{(n)}}+\sqrt{X_t^{(n)}}d(W^{(n)}_t)^*\sqrt{X^{(n)}_t}+\alpha X_t^{(n)}dt,\qquad\text{$X^{(n)}_0\in \mathcal{H}_n^+$}.
\end{align}
We will also study its real-valued analogue on $\mathcal{S}_n^+$
\begin{align}
\label{eq:GMBM:SDE:real}
dX^{(n)}_t=\sqrt{X_t^{(n)}}dB^{(n)}_t\sqrt{X_t^{(n)}}+\sqrt{X_t^{(n)}}d(B^{(n)}_t)^T\sqrt{X^{(n)}_t}+\alpha X_t^{(n)}dt,\qquad X^{(n)}_0\in \mathcal{S}_n^+\/.
\end{align}
Note that for $n=1$ the above-given SDE reduces to 
$$
  dX_t  = 2X_tdB_t+\alpha X_tdt\/,\quad X_0=x_0\geq 0\/,
$$
with one-dimensional Brownian motion $(B_t)$. Its unique solution is the geometric Brownian motion $X_t = x_0\exp(2B_t+(\alpha-2)t)$. We begin this section with the following result. 
\begin{theorem}
\label{thm:GMBM:1}
 If $\mu_0^{(n)}\Rightarrow \delta_a$, $a>0$, then the sequence of the empirical measure-valued processes 
$$
\mu_t^{(n)}(dx) = \frac{1}{n}\sum_{i=1}^n \delta_{\lambda_i^{(n)}}(dx)
$$
 defined for a solution $(X_t^{(n)})$ to \eqref{eq:GMBM:SDE} converges in probability in the space $\mathcal{C}(\R_+,\textrm{Pr}(\R))$ to $(\mu^{geo}_t(dx))_{t\geq 0}$. For every $t\geq 0$ the measure $\mu_t^{geo}$ is supported on $(0,\infty)$ and $(\mu_t^{geo})_{t\geq 0}$ is the unique solution to 
\begin{eqnarray}
\label{eq:GMBM:integral}
  \langle\mu_{t},f\rangle &=& f(a)+\alpha \int_{0}^{t}ds\int_{\R_+}xf'(x)\mu_s(dx)+\beta \int_{0}^{t}ds\int_{\R_+^{2}}
	\frac{f^{\prime}(x)-f^{\prime}(y)}{x-y}xy\mu_{s}(dx)\mu_{s}(dy),
\end{eqnarray}
with $\beta=2$. The analogous result holds for solutions to \eqref{eq:GMBM:SDE:real} with $\beta=1$.
\end{theorem}

\begin{proof}
   We begin with considering the determinants $e_n(t) = \det(X_t^{(n)}(t))=\lambda_1^{(n)}(t)\cdot\ldots\cdot\lambda_n^{(n)}(t)$, which by \eqref{eq:sympol:SDE} are described by the following SDE
	\begin{eqnarray*}	  
	  de_n = 2e_ndB_t+e_n(\alpha n-\beta (n-1))dt 
	\end{eqnarray*}
	for every $t<T_0^{(n)}$, with $T_0^{(n)}=\inf\{t\geq 0: \lambda_1^{(n)}(t)=0\}$. Since $e_n(0)$ is assumed to be positive and $e_n$ is the geometric Brownian motion, we get $T_0^{(n)}=\infty$ a.s. Consequently, the measures $\mu_t^{(n)}$ are supported in $(0,\infty)$ for every $t\geq 0$. To finish the proof, it is enough to show that there exists the unique solution to \eqref{eq:GMBM:integral}, i.e. the general \eqref{limit} with $g(x)=\sqrt{|x|}$, $h(x)=\sqrt{|x|}$, $b(x)=\alpha\, x$ for measures $\mu_t$ with supports in $\R_+$. Denoting by $m_k(t)$ the $k$-th moment of a solution $\mu_t$ to \eqref{eq:GMBM:integral}, we can easily get that 
	\begin{eqnarray*}
		 m_k(t) &=& a^k+\alpha k \int_0^t m_k(s)ds+\beta k\sum_{i=0}^{k-2}\int_0^t m_{i+1}(s)m_{k-1-i}(s)ds\/,
	\end{eqnarray*}
	for $k=0,1,\ldots$ and $t\geq 0$. Differentiating both sides with respect to $t$ we obtain the equivalent system of first order recurrence differential equations
	\begin{eqnarray*}
	   m_k'(t) &=& \alpha k\, m_k(t)+\beta k\, \sum_{i=0}^{k-2} m_{i+1}(t)m_{k-1-i}(t)\/, \qquad m_k(0)=a^k\/,
	\end{eqnarray*}
	which can be solved to show that 
	$$
	   m_k(t) = a^k w_k(t\beta) e^{k\alpha t}\/,\quad t\geq 0\/,k=0,1,2\ldots
	$$
	Here $w_k(x)$ denote polynomials of degree $(k-1)$ uniquely defined by
		\begin{eqnarray}
		 \label{eq:GMBM:poly}
	   w_k'(x) =  k\sum_{i=0}^{k-2}w_{i+1}(x)w_{k-1-i}(x)\/,\quad  w_k(0)&=&1\/,
	\end{eqnarray}
	where we use the convention that the sum over an empty set is equal to zero. This shows that the solutions to \eqref{eq:GMBM:integral} have the same moments $m_k(t)$. Moreover, we have
	\begin{equation}
	\label{eq:GMBM:polybound}
	w_k(x)\leq k!\,9^{k-1}\,(1+x)^{k-1}\/,\quad k=1,2,\ldots
	\end{equation}
	for every $x\geq 0$. Since $w_1(x)=1$, \eqref{eq:GMBM:polybound} is obvious for $k=1$. Assuming that \eqref{eq:GMBM:polybound} holds for every $l<k$, $k\geq 2$, we use \eqref{eq:GMBM:poly} to get
	\begin{eqnarray*}
	   w_k(x)-w_k(0) &=& \int_0^x w_k'(y)dy \leq k\sum_{i=0}^{k-2}\int_0^x w_{i+1}(y)w_{k-1-i}(y)dy\\
		&\leq& k\,9^{k-2} \sum_{i=0}^{k-2}(i+1)!(k-1-i)!\int_0^x (1+y)^{k-2}dy\leq {k! 9^{k-2}(1+x)^{k-1}}\sum_{i=0}^{k-2}\frac{1}{\binom{k-1}{i}}
	\end{eqnarray*}
	and the result follows from the mathematical induction and the fact that $\sum_{i=0}^{k-2}1/\binom{k-1}{i}\leq \sum_{i=0}^{k-1}1/\binom{k-1}{i}\leq 3$ and $1+3k! 9^{k-2}(1+x)^{k-1}\leq k! 9^{k-1}(1+x)^{k-1}$. As a consequence we get
	$$
	  \limsup_{k\to \infty} \sqrt[k]{\frac{m_k(t)}{k!}}\leq 9a(1+t\beta)\, e^{\alpha t}<\infty\/,
	$$
	which implies that $\mu_t$ is uniquely characterized by its moments. This ends the proof. 
\end{proof}
	
\noindent	The result can be generalized in the following way.
	
	\begin{theorem}
  \label{thm:GMBM:class}
   Let $g_n$, $h_n$ and $b_n$ be continuous functions and $(X_t^{(n)})$ be a solution to 
	$$
	  dX_t^{(n)} = g_n(X_t^{(n)})dW_t^{(n)}h_n(X_t^{(n)})+h_n(X_t^{(n)})d(W_t^{(n)})^*g_n(X_t^{(n)}) +\frac{1}{n}b_n(X_t^{{n}})dt\/,\quad X_0^{(n)}\in\mathcal{H}_n^+\/,
	$$
	where $W_t^{(n)}= n^{-1/2}W_t$ with $W_t$ being a $n\times n$ complex matrix Brownian motion. Assume that \eqref{eq:growth} holds, $g_n^2(x)\to x$, $h_n^2(x)\to x$ and $b_n(x)/n\to \alpha x $, locally uniformly on $[0,\infty)$ as $n\to \infty$. If additionally there exist $c_1(n), c_2(n), c_3(n)>0$ such that for every $x\geq 0$ we have
	\begin{equation}
	   \label{eq:GMBM:bounds}
	   g_n^2(x) \leq c_1(n)x\/,\quad
		 h_n^2(x) \leq  c_2(n)x\/,\quad
		 b_n(x) \geq c_3(n)x \/,
	\end{equation}
	 then the related empirical measure-valued process
	$$
	  \mu_t^{(n)}(dx) = \frac{1}{n}\sum_{i=1}^n \delta_{\lambda_i^{(n)}}(dx)
	$$
	tends in probability in the space $\mathcal{C}(\R_+,\textrm{Pr}(\R))$ to $(\mu_t^{geo})_{t\geq 0}$ defined in Theorem \ref{thm:GMBM:1} for $\beta=2$, whenever $\mu_0^{(n)}\Rightarrow \delta_a$, $a>0$, such that \eqref{eq:8momentbound} holds. The analogous result holds in the real-valued case with $\beta=1$.
\end{theorem}
	\begin{proof}
	   It is enough to show that the conditions \eqref{eq:GMBM:bounds} ensure that the process $X_t^{(n)}$ stays in $\mathcal{H}^{+}_n$ whenever $X_0^{(n)}\in\mathcal{H}^{+}_n$. Then the result follows from Theorem \ref{thm:main} and the proof of Theorem \ref{thm:GMBM:1}. Following the argument from the proof of Theorem \ref{thm:Wishart:class} we consider $V_n = \log e_n$ and use \eqref{eq:GMBM:bounds} in  \eqref{eq:Vn:SDE} to obtain that the drift part of $V_n$, which is equal to 
		\begin{eqnarray}
		\label{eq:driftVn_general}
		  -\frac{2}{n\,e_n^2}\sum_{i=1}^n g_n^2(\lambda_i^{(n)})h_n^2(\lambda_i^{(n)})(e_{n-1}^{\overline{\lambda_i^{(n)}}})^2
				+\frac{1}{ne_n}\left[\sum_{i=1}^n b_n(\lambda_i^{(n)})e_{n-1}^{\overline{\lambda_i^{(n)}}}-\beta\sum_{i<j}G_n(\lambda_i^{(n)},\lambda_j^{(n)})e_{n-2}^{\overline{\lambda_i^{(n)}},\overline{\lambda_j^{(n)}}}\right]
				\end{eqnarray}
			is bounded from below by
		$$		
			 -2c_1(n)c_2(n)+c_3(n)-\beta c_1(n)c_2(n)(n-1)\/,
		$$
		for every $t<T_0^{(n)}$, where $T_0^{(n)}$ is the first hitting time of zero by the first eigenvalue of $X_t^{(n)}$. This together with the McKean's argument gives $T_0^{(n)}=\infty$ a.s. which ends the proof.
	\end{proof}
	

\section{Jacobi processes}
\label{sec:Jacobi}
The last presented universality class relates to the Jacobi processes, i.e. solutions to the following SDE derived by Doumerc in \cite{bib:Doumerc:2005} (see also \cite{bib:Demni:2010}) 
\begin{align}
\label{eq:Jacobi:hermitian}
dX^{(n)}_t=\sqrt{|X^{(n)}_t|}dW^{(n)}_t\sqrt{|I_n-X_t|}+\sqrt{|I-X^{(n)}_t|}d(W^{(n)}_t)^*\sqrt{|X^{(n)}_t|}+((p+q)X^{(n)}_t+pI_n)dt,
\end{align}
with initial condition $X_0^{(n)}\in\mathcal{H}_n^+$. The real-valued case is described by the corresponding SDE
\begin{align}
\label{eq:Jacobi:real}
dX^{(n)}_t=\sqrt{|X^{(n)}_t|}dB^{(n)}_t\sqrt{|I_n-X_t|}+\sqrt{|I-X^{(n)}_t|}d(B^{(n)}_t)^T\sqrt{|X^{(n)}_t|}+((p+q)X^{(n)}_t+p I_n)dt,
\end{align}
with the starting point $X_0^{(n)}\in\mathcal{S}_n^+$. The two-dimensional parameter $(p,q)=(p(n),q(n))$ is called the dimension of the Jacobi process being the unique solution to \eqref{eq:Jacobi:hermitian} or \eqref{eq:Jacobi:real}. However it is usually assumed that $p\wedge q\geq n-1+2/\beta$ and this condition ensures that the eigenvalues $\lambda_1^{(n)}\leq\ldots\leq  \lambda_n^{(n)}$ remain in $[0,1]$
if they start from $[0,1]$ (see \cite{bib:Demni:2010}). In particular, we can omit the absolute values in \eqref{eq:Jacobi:hermitian}, \eqref{eq:Jacobi:real} and deduce that the empirical measures $\mu_t^{(n)}(dx)$ are supported in $[0,1]$. The corresponding integral equation \eqref{limit} for $g(x)=\sqrt{x}$, $h(x)=\sqrt{1-x}$ and $b(x) = p+(p+q)x$ has the following form
\begin{align}
\label{eq:Jacobi:integral}
\langle\mu_{t},f\rangle =& \langle\mu_{0},f\rangle+\int_{0}^{t}ds\int_{[0,1]}(p+(p+q)x)f'(x)\mu_s(dx)\\
&+\beta \int_{0}^{t}ds\int_{[0,1]^2}
\frac{f^{\prime}(x)-f^{\prime}(y)}{x-y}x(1-y)\mu_{s}(dx)\mu_{s}(dy).\nonumber
\end{align}

We begin with showing that there exists the unique solution to the above-given equation and consequently we have convergence of the empirical measure-valued processes. 

\begin{theorem}
\label{thm:Jacobi:1}
  If $\mu_0^{(n)}$ is supported in $[0,1]$ for every $n\in\N$ and $\mu_0^{(n)}\Rightarrow \delta_a$, $a>0$, then the sequence of the empirical measure-valued processes 
$$
\mu_t^{(n)}(dx) = \frac{1}{n}\sum_{i=1}^n \delta_{\lambda_i^{(n)}}(dx)
$$
 defined for a solution $X_t^{(n)}$ to \eqref{eq:Jacobi:hermitian} (or \eqref{eq:Jacobi:real}) such that 
$$
  p(n)\wedge q(n)\geq n-1+2/\beta
$$ 
converges in probability in the space $\mathcal{C}(\R_+,\textrm{Pr}(\R))$ to $(\mu^{Jac}_t(dx))_{t\geq 0}$. For every $t\geq 0$ the measure $\mu_t^{Jac}$ is supported on $[0,1]$ and $(\mu_t^{Jac})_{t\geq 0}$ is the unique solution to \eqref{eq:Jacobi:integral} with $\beta=2$ (or $\beta=1$).
\end{theorem}

\begin{proof}
 The proof follows in the similar way as in Theorem \ref{thm:GMBM:1}. We consider moments $m_k(t)$ of a solution to \eqref{eq:Jacobi:integral} and show that 
\begin{align*}
  m_k(t) =& a^k +pk\int_0^t m_{k-1}(s)ds+(p+q)k\int_0^t m_k(s)ds\\
	&+\beta\,k\sum_{i=0}^{k-2}\int_0^t\left(  m_{i+1}(s)m_{k-2-i}(s)+m_{i+1}(s)m_{k-1-i}(s)\right)ds\/,
\end{align*}
for every $k=0,1,2,\ldots$. It implies, in particular, that all the solutions have the same moments and since we only consider measures $\mu_t$ with bounded supports (in $[0,1]$), they are uniquely determined by its moments. Thus the uniqueness of solutions to \eqref{eq:Jacobi:integral} follows and the convergence is the consequence of Theorem \ref{thm:main}.
\end{proof}

\noindent Moreover, using Theorem \ref{thm:main} we can generalize the result as follows. 

\begin{theorem}
	\label{thm:Jacobi:class}
	Let $g_n$, $h_n$ and $b_n$ be continuous functions and $X_t^{(n)}$ be a solution to 
	\begin{equation}
	\label{eq:Jacobi:general}
	dX_t^{(n)} = g_n(X_t^{(n)})dW_t^{(n)}h_n(X_t^{(n)})+h_n(X_t^{(n)})d(W_t^{(n)})^*g_n(X_t^{(n)}) +\frac{1}{n}b_n(X_t^{{n}})dt\/,\quad X_0^{(n)}\in\mathcal{H}_n^+\/,
	\end{equation}
	where $W_t^{(n)}= n^{-1/2}W_t$ with $W_t$ being an $n\times n$ complex matrix Brownian motion and $X_0^{(n)}\in\mathcal{H}_n$ such that $0<\lambda_1^{(n)}(0)\leq \ldots\leq \lambda_n^{(n)}(0)<1$. Assume that \eqref{eq:growth} holds, $g_n^2(x)\to x$, $h_n^2(x)\to 1-x$ and $b_n(x)/n\to (p+q)x+q$, uniformly on $[0,1]$ as $n\to \infty$. If additionally there exist $c_1(n), c_2(n), c_3(n)>0$ such that for every $x\in[0,1]$ we have
	\begin{equation}
	\label{eq:Jacobi:bounds}
	g_n^2(x) \leq c_1(n)x\/,\quad
	h_n^2(x) \leq  c_2(n)(1-x)\/,\quad
	b_n(x) \geq c_3(n) 
	\end{equation}
	and $c_3(n)\geq c_1(n)c_2(n)(\beta(n-1)+2)$ holds, then the related empirical measure-valued process
	$$
	\mu_t^{(n)}(dx) = \frac{1}{n}\sum_{i=1}^n \delta_{\lambda_i^{(n)}}(dx)
	$$
	tends in probability in the space $\mathcal{C}(\R_+,\textrm{Pr}(\R))$ to $(\mu_t^{Jac})_{t\geq 0}$ defined in Theorem \ref{thm:Jacobi:1}, whenever $\mu_0^{(n)}$ is supported in $[0,1]$ and $\mu_0^{(n)}\Rightarrow \delta_a$, $a>0$. The analogous result holds in the real-valued case with $\beta=1$.
\end{theorem}

\begin{proof}
To show that $0<\lambda_i^{(n)}(t)\leq \ldots\lambda_n^{(n)}(t)<1$ for every $t\geq 0$ we consider $V_n(t)=\log e_n(t)$ and $\overline{V}_n(t)=\prod_{i=1}^n (1-\lambda_i^{(n)})$, which are well-defined for $t<T_1^{(n)}\wedge T_0^{(n)}$, where
$$
T_0^{(n)}:=\inf\{t\geq 0:\lambda_1^{(n)}(t)=0\},\quad\text{and}\quad T_1^{(n)}:=\inf\{t\geq 0:\lambda_n^{(n)}(t)=1\}.
$$
We use \eqref{eq:Jacobi:bounds} to find the lower bound of the drift term of $V_n$ for $t<T_1^{(n)}\wedge T_0^{(n)}$ of the following form
\begin{align*}
   \textrm{drift}[V_n]&\geq \frac{1}{ne_n}\left[c_3(n)e_{n-1}+c_1(n)c_2(n)(e_{n-1}-ne_n)(\beta(n-1)+2)\right]\\
	& \geq c_3(n)+\frac{e_{n-1}-ne_n}{ne_n}(c_3(n)-c_1(n)c_2(n)(\beta(n-1)+2))\geq c_3(n)\/,
\end{align*}
where the last inequality follows from the fact that $e_{n-1}-ne_n\geq 0$ whenever the eigenvalues are in $[0,1]$. To deal with $\overline{V}_n$ observe that the process $Y_t^{(n)}=I_n-X_t^{(n)}$ is a solution to \eqref{eq:Jacobi:general} with $\overline{g_n}(x)=g_n(1-x)$, $\overline{h_n}(x)=h_n(1-x)$ and $\overline{b_n}(x)=b_n(1-x)$ with the initial point $Y_0^{(n)}$ having all the eigenvalues in $(0,1)$. Since \eqref{eq:Jacobi:bounds} reads then as
\begin{equation*}
  	\overline{g_n}^2(x) \leq c_1(n)(1-x)\/,\quad
	\overline{h_n}^2(x) \leq  c_2(n)x\/,\quad
	\overline{b_n}(x) \geq c_3(n)
\end{equation*}
and the same arguments as above in the case of $V_n$ gives that the drift part of $\overline{V}_n$ is bounded from below by $c_3(n)$. Consequently, by the McKean's argument, we obtain that $T_0^{(n)}\wedge T_1^{(n)}=\infty$ a.s. Thus the measures $\mu_t^{(n)}$ are supported in $[0,1]$ and the result follows from Theorem \ref{thm:main} and Theorem \ref{thm:Jacobi:1}.
\end{proof}

\section{Free probability}
\label{sec:Free}

Let us consider a non-commutative $W^*$-probability space $(\mathcal{A},\E)$, that is, a Von Neumann operator algebra $\mathcal{A}$ with a faithful normal trace $\E$. We denote the operator norm by $\|X\|_{\mathcal{A}}$, and the $L^{2}$-norm by $\|X\|_2:=\sqrt{E(XX^*)}$.
\par The \textit{spectral distribution} of a self-adjoint operator $X\in\mathcal{A}$ is a probability measure $\mu$ on $\R$ such that
\[
\E(X^k)=\int_{\R}x^k\mu(dx).
\]
\par Let $\bar{A}_i$ denote an arbitrary element of an algebra $\mathcal{A}_i$. The sub-algebras $\mathcal{A}_1,\mathcal{A}_2,\dots,\mathcal{A}_n$ of $\mathcal{A}$ (and operators that generate them) are said to be \textit{freely independent} or \textit{free}, if the following condition holds:
\[
\E(\bar{A}_{i_1}\dots\bar{A}_{i_m})=0,
\] 
provided that $\E(\bar{A}_{i_s})=0$ and $i_{s+1}\not=i_s$ for every $s$. Two particular consequences are
\begin{itemize}
\item[(i)] $\E(AB)=\E(A)\E(B)$ if $A$ and $B$ are free.
\item[(ii)] 
\begin{equation}\label{freeness}
\E(AX_1AX_2)=\E(A^2)\E(X_1)\E(X_2),
\end{equation}
if $A$ is free from $X_1$ and $X_2$ and $\E(A)=0$.
\end{itemize}
The \textit{free Brownian motion}, is a family of operators $W:=\{W_t:t\geq0\}$ that satisfies the following properties:
\begin{itemize}
\item[(i)] $W_0=0$.
\item[(ii)] The increments of $W$ are free; i.e. $W_t-W_s$ is free from the subalgebra $\mathcal{A}_s$ which is generated by all $W_r$ with $r\leq s$.
\item[(iii)] The spectral distribution of $W_t-W_s$ is semicircle with zero expectation and variance $t-s$. 
\end{itemize}
For an adapted biprocess $N=a\otimes b:\R_+\to\mathcal{A}\otimes\mathcal{A}$ we will consider integrals with respect to a free Brownian motion of the form
\[
\int_0^{\infty}a_t(dW_t)b_t.
\]
Which satisfy Burkholder-Gundy type inequalities in the operator norms
\begin{equation}
\label{L_2norm}
\left\|\int_0^{\infty}a_t(dW_t)b_t\right\|_{\mathcal{A}}\leq 2\sqrt{2}\left(\int_0^{\infty}\|N_s\|^2_{\mathcal{A}\otimes\mathcal{A}}ds\right)^{1/2}.
\end{equation}
For more details on free stochastic calculus we refer to \cite{bib:BianeSpiecher1998}.

Following Capitaine and Donati-Martin \cite{bib:CapitaineDonatiMartin2005} we define a complex free Brownian motion $Z$. To this end we will consider $Z:=(U,V)$ a 2-dimensional $(\mathcal{A}_t)$-free Brownian motion in $(\mathcal{A},E)$. And we define the complex free Brownian motion as
\[
Z:=\frac{(X+iY)}{\sqrt{2}}.
\]
We will consider of the integral with respect to the complex free Brownian motion $Z$, so for any  adapted biprocess $N=a\otimes b:\R_+\to\mathcal{A}\otimes\mathcal{A}$ we define
\begin{equation}\label{def_csi}
\int_0^{\infty}a_t(dZ_t)b_t:=\frac{1}{\sqrt{2}}\left[\int_0^{\infty}a_t(dU_t)b_t+i\int_0^{\infty}a_t(dV_t)b_t\right].
\end{equation}
Following \cite{bib:CapitaineDonatiMartin2005}, the integral with respect to complex free Brownian motion satisfies for any adapted processes $a_t,b_t$ and $c_t$ the following properties:
\begin{itemize}
\item[(i)] \begin{equation}\label{prop0}\left(\int_0^{\infty}a_t dZ_tb_t\right)^*=\int_0^{\infty}b^*_t dZ^*_ta^*_t.\end{equation}
\item[(ii)] \begin{equation}\label{prop1}\E\left(\int_0^{\infty}a_t dZ_tb_t\int_0^{\infty}c_t dZ_td_t\right)=0.\end{equation}
\item[(iii)] \begin{equation}\label{prop2}\E\left(\int_0^{\infty}a_t dZ_tb_t\int_0^{\infty}c_t dZ^*_td_t\right)=\int_0^{\infty}\E(a_td_t)\E(b_tc_t)dt.\end{equation}
\end{itemize}
Now we will provide a version of Lemma 3.3 in \cite{bib:Kargin2011} for the case of complex free Brownian motion.
\begin{lemma}\label{lemma:freeintegral}
Let operators $H_1$, $H_2$, and $H_3$ belong to the subalgebra $\mathcal{A}_a$ which is generated by $Z_t$ for $t\leq a$. Then
\begin{itemize}
\item[(i)] 
\[
\E\left[\left(H_1\int_a^ba_t(dZ_t)b_t\right)H_2\left(\int_a^bc_t(dZ_t)d_t\right)H_3\right]=0.
\]
\item[(ii)]
\[
\E\left[H_1\left(\int_a^ba_t(dZ_t)b_t\right)H_2\left(\int_a^bc_t(dZ^*_t)d_t\right)H_3\right]=\int_a^b\E(b_tH_2c_t)\E(d_tH_3H_1a_t)dt.
\]
\end{itemize}
\end{lemma}
As in \cite{bib:Kargin2011} the result follows by writing the integral as the limit of sums and using identity \eqref{freeness} together with the free independence between $X_t$ and $Y_t$.
\par For a given function $g:\R\to\R$, we write $g(X)$ for the spectral action of $g$ on $X\in\mathcal{A}$. Let us consider a solution $X=\{X_t:t\geq0\}$ to the following free stochastic differential equation 
\begin{equation}
\label{freediffusion}
dX_t=g(x_t)dZ_th(X_t)+h(X_t)dZ^*_tg(X_t)+b(X_t)dt,\qquad\text{$X_0\in\mathcal{A}$,}
\end{equation}
where the functions $g,h,b:\R\to\R$ act spectrally on $X$ as described above, and $Z$ is a complex free Brownian motion. 

We say that a function $f:\R\to\C$ is locally operator-Lipschitz continuous, if it is measurable, locally bounded, and if for $K>0$ it exists $C_K$, such that
\[
\|f(X)-f(Y)\|_{\mathcal{A}}\leq C_K\|X-Y\|_{\mathcal{A}},
\]
for all self-adjoint operators $X$ and $Y$ with norm less than $K$.
By a slight modification of the proof of Theorem 3.1 in \cite{bib:Kargin2011} we have that if the functions $g,h$ and $b$ are locally operator-Lipschitz continuous and $\overline{X}$ is bounded in operator norm then there exists $t_0>0$ and a family operators $(X_t)_{t\geq0}$ defined for all $t\in[0,t_0)$ and bounded in operator norm, such that $X_0=\overline{X}$ and $X_t$ is a unique solution of \eqref{freediffusion} for $t<t_0$.
\par  Let us denote by $R_t$ the resolvent of $X_t$ given by
\begin{align}
R_t(z):=(X_t-z)^{-1},\qquad\text{$z\in\C^+$,}
\end{align}
and by $r_t$ to its expectation
\begin{align}
r_t(z):=\E[(X_t-z)^{-1}],\qquad\text{$z\in\C^+$.}
\end{align}
Now we will characterize the solution $X$ to \eqref{freediffusion} in terms of its Cauchy transform. To this end we provide the following version of Theorem 3.2 in \cite{bib:Kargin2011}. The proof of this result is quite similar to that in \cite{bib:Kargin2011}, but we include the proof for sake of completeness.
\begin{theorem}\label{lem:cauchu_fd}
Assume that  $g$, $h$ and $b$ are locally operator-Lipschitz continuous, and let $X$ be an operator bounded solution to \eqref{freediffusion} for all $t\in[0,t_0)$. Let $R_t$ and $r_t$ denote the resolvent of $X_t$  and the expectation of the resolvent, respectively, and let $g_t:=g(X_t)$, $h_t:=h(X_t)$, and $b_t:=b(X_t)$. Then, for all $t\in[0,t_0)$ and $z\in\C^+$,
\begin{equation}\label{cauchy_fd}
\frac{dr_t(z)}{dt}=-\E[b_t(R_t(z))^2]+\E[g^2_tR_t(z)]\E[h^2_t(R_t(z))^2]+\E[h^2_tR_t(z)]\E[g^2_t(R_t(z))^2].
\end{equation}
\end{theorem}
\begin{proof}
Following the proof of Theorem 3.2 in \cite{bib:Kargin2011} we will use the following notation:
\begin{align*}
A=\int_t^{\Delta t}b_sds,\qquad
B=\int_t^{\Delta t} g_s(dZ_s)h_s,\qquad\text{and }\qquad C=\int_t^{\Delta t} h_s(dZ^*_s)g_s.
\end{align*}

We note that using the fact that $g$, $h$ and $b$ are locally operator-Lipschitz continuous, and the fact $\sup_{0\leq t\leq t_0}\|X_t\|_{\mathcal{A}}<\infty$, together with \eqref{prop2} we obtain that
\begin{equation}\label{bounds}
\|A\|_{2}=O(\Delta t),\qquad\|B\|_{2}=O(\sqrt{\Delta t}),\text{and}\qquad \|C\|_{2}=O(\sqrt{\Delta t}).
\end{equation}
\par By using twice the resolvent identity we have
\begin{align}\label{resolvent1}
R_{t+\Delta t}-R_t&=-R_{t+\Delta t}(A+B+C)R_t\notag\\
&=-R_tAR_t-R_t(B+C)R_t+R_{t+\Delta t}(A+B+C)R_t(A+B+C)R_t.
\end{align}
And we note that
\begin{equation}\label{op:bound1}
\|R_{t+\Delta t}AR_tAR_t+R_{t+\Delta t}AR_t(B+C)R_t+R_{t+\Delta t}(B+C)R_tAR_t\|_{2}=o(\Delta t).
\end{equation}
Also using \eqref{bounds},
\[
\|R_{t+\Delta t}-R_t\|_{2}=O(\sqrt{\Delta t}),
\]
which gives
\begin{equation}\label{op:bound2}
\|R_{t+\Delta t}(B+C)R_t(B+C)R_t-R_{t}(B+C)R_t(B+C)R_t\|_{2}=o(\Delta t).
\end{equation}
Hence using \eqref{op:bound1} and \eqref{op:bound2} in \eqref{resolvent1} we get
\begin{equation}\label{resolvent2}
\E(R_{t+\Delta t}-R_t)=\E(-R_tAR_t-R_t(B+C)R_t+R_{t}(B+C)R_t(B+C)R_t).
\end{equation}
Now we note that
\begin{equation}
\label{op:bound3}
\int_t^{t+\Delta t}b_sds=b_t\Delta t+o(\Delta t),
\end{equation}
and 
\begin{equation}\label{op:bound4}
\E\left[\int_t^{t+\Delta t}G_tg_s(dZ_s)h_sG_t\right]=0,
\qquad \E\left[\int_t^{t+\Delta t}G_th_s(dZ^*_s)g_sG_t\right]=0.
\end{equation}
Finally by Lemma \ref{lemma:freeintegral} and the fact that $g$ and $h$ are locally operator-Lipschitz continuous we get that
\begin{align*}
\E\Big[R_t\left(\int_t^{\Delta t}g_s(dZ_s)h_s\right)&R_t\left(\int_t^{\Delta t}g_s(dZ_s)h_s\right)R_t\Big]\\&=\E\left[R_t\left(\int_t^{\Delta t}h_s(dZ^*_s)g_s\right)R_t\left(\int_t^{\Delta t}h_s(dZ^*_s)g_s\right)R_t\right]=0,
\end{align*}
and
\begin{align*}
\E\Big[R_t\left(\int_t^{\Delta t}g_s(dZ_s)h_s\right)R_t\left(\int_t^{\Delta t}h_s(dZ^*_s)g_s\right)R_t\Big]&=(\Delta_t)\E[h_tR_th_t]\E[g_tR_t^2g_t]+o(\Delta t),\\
\E\left[R_t\left(\int_t^{\Delta t}h_s(dZ^*_s)g_s\right)R_t\left(\int_t^{\Delta t}g_s(dZ_s)h_s\right)R_t\right]&=(\Delta_t)\E[g_tR_tg_t]\E[h_tR_t^2h_t]+o(\Delta t).
\end{align*}
These identities, together with the properties of the trace imply
\begin{align}
\label{op:bound6}
\E(R_t(B+C)R_t(B+C)R_t))=(\Delta t)\left(\E[h_tR_th_t]\E[g_tR_t^2g_t]+\E[g_tR_tg_t]\E[h_tR_t^2h_t]\right)+o(\Delta t).
\end{align}
Hence by using \eqref{op:bound3}, \eqref{op:bound4}, and \eqref{op:bound6} in \eqref{resolvent2} we get
\[
r_{t+\Delta t}-r_t=(\Delta t)\left(-\E[b_tR_t^2]+\E[g_t^2R_t]\E[h_t^2R_t^2]+\E[h_t^2R_t]\E[g_t^2R_t^2]\right)+o(\Delta t).
\]
Dividing by $\Delta t$ and taking $\Delta t\to0$ we obtain the result.
\end{proof}
\begin{remark}\label{rem:conv_free_diff}
For each $t<t_0$ let us denote by $\tilde{\mu}_t$ the spectral distribution of $X_t$, which is the unique solution to \eqref{freediffusion}. Then we can rewrite \eqref{cauchy_fd} in the following form
\begin{align}\label{cauchy_fd_2}
\frac{d}{dt}\int_{\R}\frac{1}{(x-z)}\tilde{\mu}_t(dx)&=-\int_{\R}\frac{b(x)}{(x-z)^2}\tilde{\mu}_t(dx)+\int_{\R}\frac{g^2(x)}{(x-z)}\tilde{\mu}_t(dx)\int_{\R}\frac{h^2(y)}{(y-z)^2}\tilde{\mu}_t(dy)
\notag\\&+\int_{\R}\frac{g^2(x)}{(x-z)^2}\tilde{\mu}_t(dx)\int_{\R}\frac{h^2(y)}{(y-z)}\tilde{\mu}_t(dy)
\end{align}
Therefore, if $\mu_0=\tilde{\mu}_0$ and there is a unique solution to \eqref{cauchy_fd_2}, then by \eqref{ct2} we have that $\tilde{\mu}_t=\mu_t$ for all $t<t_0$. Hence, under the conditions of Theorems \ref{thm:main} and \ref{lem:cauchu_fd}, the family of the measure-valued processes $\{(\mu_{t}^{(n)})_{t<t_0}:n\geq1\}$ converges to the law $(\tilde{\mu}_{t})_{t<t_0}$ of the free diffusion $X$ given by \eqref{freediffusion}.
\end{remark}
\subsection{Examples.}
\subsubsection{Free linear Brownian motion.}
Let su consider $X_t$ the solution to the following free stochastic differential equation
\begin{align}\label{free_LBM}
dX_t=\theta \mathbf{1} dt+\sigma dW_t,\qquad\text{$X_0=0$,}
\end{align}
where $\theta\in\R$, $\sigma>0$, $\mathbf{1}$ denotes the unit element in the algebra, and $W_t$ is a real free Brownian motion. This is a particular case of \eqref{freediffusion} with $g(x)=h(x)=\sigma/2$, and $b(x)=\theta$ for $x\in \R$. Then if we denote by $\mu_t$ the spectral distribution of the process $X_t$ for $t>0$, then by Remark \ref{rem:conv_free_diff} the Cauchy transform $r_t$ of $\mu_t$ is given as the solution of the following partial differential equation
\begin{align*}
\frac{dr_t(z)}{dt}&=\sigma^2r_t(z)\frac{dr_t(z)}{dz}-\theta \frac{dr_t(z)}{dz},\\
r_0(z)&=\frac{1}{z}, \qquad\text{$t>0, \ z\in\C^+.$}
\end{align*}
The uniqueness to the previous can be obtained by similar methods as those in Section 4 of \cite{bib:RogersShi:1993}, and can be solved by the method of characteristics and it is given by
\begin{align}\label{cauchy_trans_lBM}
r_t(z)=\frac{-(z-\theta t)+\sqrt{(z-\theta t)^2-4\sigma^2 t}}{2\sigma^2t},\text{$t>0, \ \Im(z)\not= 0.$}
\end{align}
For each $t>0$, \eqref{cauchy_trans_lBM} corresponds to the Cauchy transform of a semicircular distribution with mean $\theta t$ and variance $\sigma^2t$ and it is explicitly given by
\begin{align*}
\mu_t=\frac{1}{\pi \sigma t} \sqrt{2t-(x-\theta t)^2}\ind_{\{|x-\theta t|\leq \sqrt{2\sigma t}\}}dx.
\end{align*}
Hence, by an application of Theorem \ref{thm:main} together with Remark \ref{rem:conv_free_diff} we obtain:  
\begin{proposition}
	Let $g_n$, $h_n$ and $b_n$ be continuous functions and $(X_t^{(n)})$ be a solution to 
	$$
	dX_t^{(n)} = g_n(X_t^{(n)})dW_t^{(n)}h_n(X_t^{(n)})+h_n(X_t^{(n)})d(W_t^{(n)})^*g_n(X_t^{(n)}) +\frac{1}{n}b_n(X_t^{{n}})dt\/,\quad X_0^{(n)}\in\mathcal{H}_n^+\/,
	$$
	where $W_t^{(n)}= n^{-1/2}W_t$ with $W_t$ being a $n\times n$ complex matrix Brownian motion. Assume that \eqref{eq:growth} holds, $g_n^2(x)\to \sigma/\sqrt{2}$, $h_n^2(x)\to \sigma/\sqrt{2}$ and $b_n(x)/n\to \theta $, locally uniformly as $n\to \infty$. If we assume that $\mu_0^{(n)}\Rightarrow \delta_0$, then the sequence empirical measure-valued process
	$$
	\mu_t^{(n)}(dx) = \frac{1}{n}\sum_{i=1}^n \delta_{\lambda_i^{(n)}}(dx)
	$$
	tends in probability in the space $\mathcal{C}(\R_+,\textrm{Pr}(\R))$ to the spectral distribution $(\mu_t)_{t\geq0}$ of the free linear Brownian motion $X$ which is solution to \eqref{free_LBM}.
\end{proposition}
\subsubsection{Free Ornstein-Uhlenbeck process.}
Let us consider $X_t$ the solution to the following free stochastic differential equation
\begin{equation}\label{fou}
dX_t=\theta X_tdt+\sigma dW_t,\qquad\text{$X_0=0$,}
\end{equation}
where $\theta,\sigma\in\R$ and $W_t$ is real free Brownian motion. For each $t\geq0$ let us denote by $\mu_t$ the spectral distribution of $X_t$. This equation takes the form \eqref{freediffusion} with $g(x)=h(x)=\sigma/\sqrt{2}$, and $b(x)=\theta x$ for $x\in\R$.
Hence, by \eqref{cauchy_fd_2} the Cauchy transform $r_t$ of the law $\mu_t$ satisfies the following partial differential equation
\begin{align*}
\frac{dr_t(z)}{dt}&=-(\theta z-\sigma^2r_t(z))\frac{dr_t(z)}{dz}-\theta r_t(z),\\
r_0(z)&=\frac{1}{z}, \qquad\text{$t>0, \ z\in\C^+.$}
\end{align*}
The previous equation has a unique solution that can be found by the method of characteristics (see Section 3.3.1 in \cite{bib:Kargin2011}) and it is given by 
\begin{align*}
r_t(z)=\frac{\theta }{\sigma^2(e^{2\theta t}-1)}\sqrt{z^2-\frac{2\sigma^2(e^{2\theta t}-1)}{\theta}}-z,\qquad\text{$t>0, \ z\in\C^+.$}
\end{align*}
Therefore for each fixed $t>0$, $\mu_t$ has a semicircle distribution with radius given by 
\begin{align*}
\sqrt{\frac{2\sigma^2(e^{2\theta t}-1)}{\theta}}\qquad&\text{if $\theta\geq0$},\\
\sqrt{\frac{2\sigma^2(1-e^{-2|\theta| t})}{|\theta|}}\qquad&\text{if $\theta<0$}.
\end{align*}
Therefore, by Theorem \ref{thm:main} together with Remark \ref{rem:conv_free_diff} we have the following result:
\begin{proposition}
Let $g_n$, $h_n$ and $b_n$ be continuous functions and $(X_t^{(n)})$ be a solution to 
$$
dX_t^{(n)} = g_n(X_t^{(n)})dW_t^{(n)}h_n(X_t^{(n)})+h_n(X_t^{(n)})d(W_t^{(n)})^*g_n(X_t^{(n)}) +\frac{1}{n}b_n(X_t^{{n}})dt\/,\quad X_0^{(n)}\in\mathcal{H}_n^+\/,
$$
where $W_t^{(n)}= n^{-1/2}W_t$ with $W_t$ being a $n\times n$ complex matrix Brownian motion. Assume that \eqref{eq:growth} holds, $g_n^2(x)\to \sigma/\sqrt{2}$, $h_n^2(x)\to \sigma/\sqrt{2}$ and $b_n(x)/n\to \theta x$, locally uniformly as $n\to \infty$. If we assume that $\mu_0^{(n)}\Rightarrow \delta_0$, then the sequence empirical measure-valued process
$$
\mu_t^{(n)}(dx) = \frac{1}{n}\sum_{i=1}^n \delta_{\lambda_i^{(n)}}(dx)
$$
tends in probability in the space $\mathcal{C}(\R_+,\textrm{Pr}(\R))$ to the spectral distribution $(\mu_t)_{t\geq0}$ of the free Ornstein-Uhlenbeck process $X$ which is solution to \eqref{fou}.
\end{proposition}

\end{document}